\documentclass[12pt, reqno]{amsart}
\usepackage{amsmath, amsthm, amscd, amsfonts, amssymb, graphicx, color}
\usepackage[bookmarksnumbered, colorlinks, plainpages]{hyperref}
\hypersetup{colorlinks=true,linkcolor=red, anchorcolor=green, citecolor=cyan, urlcolor=red, filecolor=magenta, pdftoolbar=true}

\newtheorem{theorem}{Theorem}[section]
\newtheorem{lemma}[theorem]{Lemma}

\newtheorem{corollary}[theorem]{Corollary}
\theoremstyle{definition}
\newtheorem{definition}[theorem]{Definition}

\theoremstyle{remark}
\newtheorem{remark}[theorem]{Remark}
\numberwithin{equation}{section}

\begin{document}

\setcounter{page}{1}

\title[Weighted local Hardy spaces]{The Atomic Characterization of Weighted Local Hardy Spaces and Its Applications}

\author[X. Chen]{Xinyu Chen}
\address{Xinyu Chen, School of Science, Nanjing University of Posts and Telecommunications, Nanjing 210023, China}
\email{\textcolor[rgb]{0.00,0.00,0.84}{chenxinyu1130@126.com}}

\author[J. Tan]{Jian Tan*}

\address{Jian Tan(Corresponding author), School of Science, Nanjing University of Posts and Telecommunications, Nanjing 210023, China}
\email{\textcolor[rgb]{0.00,0.00,0.84}{tj@njupt.edu.cn; tanjian89@126.com}}


\let\thefootnote\relax\footnote{* Corresponding author.}




\begin{abstract}
The purpose of this paper is to obtain atomic decomposition characterization of the weighted local Hardy space $h_{\omega}^{p}(\mathbb {R}^{n})$ with $\omega\in A_{\infty}(\mathbb {R}^{n})$. We apply the discrete version of Calder\'on's identity and the weighted Littlewood--Paley--Stein theory to prove that $h_{\omega}^{p}(\mathbb {R}^{n})$ coincides with the weighted$\text{-}(p,q,s)$ atomic local Hardy space $h_{\omega,atom}^{p,q,s}(\mathbb {R}^{n})$ for $0<p<\infty$. The atomic decomposition theorems in our paper improve the previous atomic decomposition results of local weighted Hardy spaces in the literature. As applications, we derive the boundedness of inhomogeneous Calder\'on--Zygmund singular integrals and local fractional integrals on weighted local Hardy spaces.
\end{abstract} 
\maketitle

\section{Introduction}
The real-variable theory of global Hardy spaces on $\mathbb {R}^{n}$ was essentially developed by Stein and Weiss \cite{1960On} and systematically studied by Fefferman and Stein \cite{fefferman1972}. Hardy spaces $H^{p}(\mathbb {R}^{n})$ serve as a substitute for $L^{p}(\mathbb {R}^{n})$ when $p\leq1$. However, the principle of $H^{p}(\mathbb {R}^{n})$ breaks down at some key points, for example, pseudo-differential operators are not bounded on $H^{p}$. Hence, Goldberg in \cite{goldberg1979local} introduced the class of local Hardy spaces $h^{p}(\mathbb {R}^{n})$ with $p\in (0,1]$. Moreover, Goldberg \cite{goldberg1979local} established the maximal function characterization of $h^{p}(\mathbb {R}^{n})$ for $p\in ((n-1)/n,1]$. From then on, local Hardy spaces have become an indispensable part in terms of harmonic analysis and partial differential equations.
Then Peloso and Secco \cite{peloso2008local} obtained local Riesz transforms of local Hardy spaces and extended some characterizations of Hardy spaces $H^{p}(\mathbb {R}^{n})$ to the local Hardy spaces $h^{p}(\mathbb {R}^{n})$ for $0<p\leq1$. In 1983, Triebel \cite{1983Theory} first established the Littlewood--Paley characterization of $h^{p}(\mathbb {R}^{n})$ which is a tool to prove that $h^{p}(\mathbb {R}^{n})$ coincides with the Triebel-Lizorkin space $F_{p,2}^{0}(\mathbb {R}^{n})$. In 1981, the weighted version $h^{p}_{\omega}(\mathbb {R}^{n})$ of $h^{p}(\mathbb {R}^{n})$ with $\omega\in A_{\infty}(\mathbb {R}^{n})$ was developed by Bui \cite{qui1981weighted}. Later, Rychkov \cite{rychkov2001littlewood} extended a part of the theory of weighted local Hardy spaces to $A_{\infty}^{loc}(\mathbb {R}^{n})$ weights and obtained the Littlewood-Paley function characterization of $h_{\omega}^{p}(\mathbb {R}^{n})$. In 2012, Tang \cite{tang2012weighted} established the weighted atomic characterization of $h_{\omega}^{p}(\mathbb {R}^{n})$ with $\omega\in A_{\infty}^{loc}(\mathbb {R}^{n})$ via the local grand maximal function. It is worth pointing out that in recent years, various Hardy-type spaces were introduced and studied in \cite{nakai2012hardy,sawano2017hardy,tan2019atomic,yang2012local}.\par
As is well-known, the atomic decomposition plays an important role in the study of the boundedness of operators on Hardy-type spaces and many theories of it have been established. In 1974, Coifman \cite{coifman1974real} first introduced an atomic decomposition characterization of Hardy spaces on $\mathbb {R}$. Later, the extension to higher dimensions was obtained by Latter \cite{1978A}. In fact, the marked difference between the atomic characterization of $H^{p}(\mathbb {R}^{n})$ and $h^{p}(\mathbb {R}^{n})$ is the cancellation property of atoms. To be precise, the vanishing moment is needed only for the atoms with small supports in $h^{p}(\mathbb {R}^{n})$ while the vanishing moment is needed for all atoms in $H^{p}(\mathbb {R}^{n})$. In \cite{ding2012boundedness}, Y. Ding et al. established the atomic decomposition characterization of the weighted Hardy spaces $H^{p}_{\omega}(\mathbb {R}^{n})$ for $p\in(0,1]$ and obtained the $(H^{p}_{\omega}(\mathbb R^{n}),L^{p}_{\omega}(\mathbb R^{n}))\text{-}$boundedness for singular integrals via the discrete Calder\'on's identity and the weighted Littlewood-Paley-Stein theory. In \cite{ding2019discrete}, W. Ding et al. obtained the $L^{2}$ atomic decomposition of local Hardy spaces $h^{p}(\mathbb {R}^{n})$ for $0<p\leq1$.  Motivated by these results, we give the atomic decomposition characterization of the weighted local Hardy spaces $h_{\omega}^{p}(\mathbb {R}^{n})$ and a proof of the convergence of the atomic decomposition in both $h_{\omega}^{p}(\mathbb {R}^{n})$ and $L^{q}(\mathbb {R}^{n})$ norms for any $f\in h_{\omega}^{p}(\mathbb {R}^{n})\cap L^{q}(\mathbb {R}^{n})$. The atomic decomposition characterization in our paper provides extensions of the results in \cite{Li2021} by $w\textbf{-}(p,q,s)\textbf{-}$atom and $w\textbf{-}(p,q,s)\textbf{-}$block. In fact, we merely assume that $A_{\infty}(\mathbb {R}^{n})$ and $0<p<\infty$. Moreover, the results have a wide applicability to more general settings in that we avoid the maximal function characterization and the Calder\'on--Zymund decomposition.\par
The class of weighted local Hardy spaces $h_{\omega}^{p}(\mathbb R^{n})$ can be defined by the finiteness of the quasi-norm \cite{rychkov2001littlewood}. To be precise, let $\Phi\in\mathcal{S}(\mathbb R^{n})$ with $\int\Phi\neq0$ and $\Phi_{t}(x)=t^{-n}\Phi(\frac{x}{t})$, then
    \[
    M_{\Phi}(f)(x)=\sup\limits_{0<t<1}\vert\Phi_{t}\ast f(x)\vert. 
    \]
    Then the weighted local Hardy space $h_{\omega}^{p}(\mathbb R^{n})$ for $0<p<\infty$ and $\omega\in A_{\infty}$ is defined by
    \[
    h_{\omega}^{p}=\{f\in\mathcal{S}^{\prime}(\mathbb R^{n})\colon M_{\Phi}(f)\in L^{p}_{\omega}(\mathbb R^{n})\}
    \]
where 
\[
\left \|f\right \|_{h_{\omega}^{p}(\mathbb{R}^{n})}=\left \| M_{\Phi}(f) \right \|_{L_{\omega}^{p}(\mathbb {R}^{n})}.
\]\par
In fact, we can also define the weighted local Hardy space via the discrete Littlewood--Paley--Stein theory. Thus, we firstly recall some definitions as follows. For more details, see \cite{han1994littlewood}.
\begin{definition}
    Let $\phi_{0},\ \phi\in \mathcal{S}(\mathbb{R}^{n})$ with
\begin{equation}
    {\rm supp}\widehat\phi_{0}\subseteq\{\xi\in\mathbb{R}^{n}\colon\vert\xi\vert\leq2\};\  \widehat\phi_{0}\{\xi\}=1,\ {\rm if}\ \vert\xi\vert\leq1,\
\end{equation}
and
\begin{equation}
    {\rm supp}\widehat\phi\subseteq\{\xi\in\mathbb{R}^{n}\colon\frac{1}{2}\leq\vert\xi\vert\leq2\},
\end{equation}
and for all $\xi\in\mathbb {R}^{n}$
\begin{equation}\label{eq1.3}
    \vert\widehat{\phi_0}(\xi)\vert^{2}+ \sum\limits_{j=1}^{\infty}\vert\widehat{\phi}(2^{-j}\xi)\vert^{2}=1.
\end{equation}
\end{definition}
\noindent
Additionally, define $\phi_{j}(x)=2^{jn}\phi(2^{j}x)$ for $j\in\mathbb {N}$ and $j\geq1$. For any $j\in\mathbb {Z}$, denote $\Pi_{j}$=\{Q$\colon$Q are dyadic cubes in $\mathbb {R}^{n}$ with $l(Q)=2^{-j}$ and the left lower corners of Q are $x_{Q}=2^{-j}l$, $l\in\mathbb {Z}^{n}$\}.
By applying Fourier transform and equation ({\upshape\ref{eq1.3}}), we can obtain the continuous Calder\'on's identity \cite{ding2019discrete}:
\begin{equation}
    f(x)=\sum\limits_{j=0}^{\infty}\phi_{j}\ast\phi_{j}\ast f(x)
\end{equation}
where the series converges in $L^{q}(\mathbb {R}^{n})$, $\mathcal{S}({\mathbb {R}^{n}})$ and $\mathcal{S}^{\prime}({\mathbb {R}^{n}})$. Furthermore, we can discretize the above identity:
\[
f(x)=\sum\limits_{j=0}^{\infty}\sum\limits_{Q\in\Pi_{j}}\vert Q\vert(\phi_{j}\ast f)(x_{Q})\phi_{j}(x-x_{Q}).
\]
where the series converges in $L^{q}(\mathbb {R}^{n})$, $\mathcal{S}({\mathbb {R}^{n}})$ and $\mathcal{S}^{\prime}({\mathbb {R}^{n}})$.\par
Suppose that $\phi_{0}$, $\phi\in\mathcal{S}(\mathbb R^{n})$ satisfies (1.1)-(1.3). Based on the above reproducing formula, we give the definition of inhomogeneous Littlewood-Paley-Stein square function
\[
g(f)(x)=\left\{\sum\limits_{i\in\mathbb N}\vert\phi_{i}\ast f(x)\vert^{2}\right\}^{\frac{1}{2}}
\]
and the definition of the discrete Littlewood-Paley-Stein square function
\[
g_{d}(f)(x)=\left\{\sum\limits_{j\in\mathbb N}\sum\limits_{Q\in\Pi_{j}}\vert\phi_{j}\ast f(x_{Q})\vert^2\chi_{Q}(x)\right\}^{\frac{1}{2}}.
\]\par
Now we can give the definition of the weighted local Hardy space.
\begin{definition}
    Let $0<p<\infty,\ \omega\in A_{\infty}(\mathbb R^{n})$. Then the weighted local Hardy space $h_{\omega}^{p}(\mathbb{R}^{n})$ is defined by
\[
h_{\omega}^{p}(\mathbb{R}^{n})=\{f\in \mathcal{S}^{\prime}(\mathbb{R}^{n})\colon\left \|f\right \|_{h_{\omega}^{p}(\mathbb{R}^{n})}<\infty\},
\]
where
\[
\left \|f\right \|_{h_{\omega}^{p}(\mathbb{R}^{n})}=\left \| g_{d}(f) \right \|_{L_{\omega}^{p}(\mathbb {R}^{n})}.
\]
\end{definition}
The definitions of the atom $a$ and the block $b$ are as follows. Details are referred to \cite{tan2019atomic}.
\begin{definition}
    Let $0<p<\infty,\ 1\leq q<\infty,\ \omega\in A_{q}(\mathbb {R}^{n})$ with critical index $q_{\omega}$ and $s\in\mathbb{Z}$ fulfilling $s\geq {\rm max}\{[n(\frac {q_\omega}{p}-1)],-1\}.$ Fix a constant $C\geq1$. Then define a $\omega\text{-}(p,q,s)\text{-}$atom of $h_{\omega}^p(\mathbb R^{n})$ to be a function $a$ which is supported in a cube $Q\subseteq\mathbb R^{n}$ with $\vert Q\vert\leq C$ and satisfies
    \[
    \left \|a\right \|_{L^{q}}\leq\vert Q\vert^{\frac{1}{q}}\omega(Q)^{-\frac{1}{p}}\quad {\rm and}\quad \int_{Q}{a(x)x^{\alpha}}dx=0,\ {\rm for\ all}\ \vert\alpha\vert\leq s.
    \]
\end{definition}
\begin{definition}
    Let $0<p<\infty,\ 1\leq q<\infty,\ \omega\in A_{q}(\mathbb {R}^{n})$ with critical index $q_{\omega}$ and $s\in\mathbb{Z}$ fulfilling $s\geq {\rm max}\{[n(\frac {q_\omega}{p}-1)],-1\}.$ Fix a constant $C\geq1$. Then define a $\omega\text{-}(p,q,s)\text{-}$block of $h_{\omega}^p(\mathbb R^{n})$ to be a function $b$ which is supported in a cube $P\subseteq\mathbb R^{n}$ with $\vert P\vert>C$ and satisfies $\left \|b\right \|_{L^{q}}\leq\vert P\vert^{\frac{1}{q}}\omega(P)^{-\frac{1}{p}}$.
\end{definition}
Naturally, we can give the definition of the weighted$\text{-}(p,q,s)$ atomic local Hardy space $h_{\omega,atom}^{p,q,s}(\mathbb {R}^{n})$.
\begin{definition}
    Let $0<p<\infty,\ q_{\omega}<q<\infty,\ \omega\in A_{\infty}(\mathbb R^{n})$ with critical index $q_{\omega}$ and $s\in\mathbb{Z}$ fulfilling $s\geq {\rm max}\{[n(\frac {q_\omega}{p}-1)],-1\}.$ Then the weighted$\text{-}(p,q,s)$ atomic local Hardy space $h_{\omega,atom}^{p,q,s}(\mathbb {R}^{n})$ is defined by
\[
h_{\omega,atom}^{p,q,s}(\mathbb {R}^{n})=\left\{f\in \mathcal{S}^{\prime}(\mathbb {R}^{n})\colon f=\sum\limits_{j}\lambda_{j}a_{j}+\sum\limits_{j}\mu_{j}b_{j}\right\},
\]
where each $a_{j}$ is a $\omega\text{-}(p,q,s)\text{-}$atom and each $b_{j}$ is a $\omega\text{-}(p,q,s)\text{-}$block sastifying
\[
\left \|\sum\limits_{j=1}^{\infty}\frac{\lambda_j \chi_{Q_j}}{\omega(Q_j)^{\frac{1}{p}}}\right \|_{L_{\omega}^{p}}+\left \|\sum\limits_{j=1}^{\infty}\frac{\mu_j \chi_{P_j}}{\omega(P_j)^{\frac{1}{p}}}\right \|_{L_{\omega}^{p}}<\infty.
\]
Furthermore, we have
\[
\left \|f\right \|_{h_{\omega,atom}^{p,q,s}}=\inf\left\{\left \|\sum\limits_{j=1}^{\infty}\frac{\lambda_j \chi_{Q_j}}{\omega(Q_j)^{\frac{1}{p}}}\right \|_{L_{\omega}^{p}}+\left \|\sum\limits_{j=1}^{\infty}\frac{\mu_j \chi_{P_j}}{\omega(P_j)^{\frac{1}{p}}}\right \|_{L_{\omega}^{p}}\right\}
\]
where the infimum is taken over all decompositions $f=\sum\limits_{j}\lambda_{j}a_{j}+\sum\limits_{j}\mu_{j}b_{j}.$
\end{definition}\par
If $\omega\in A_{\infty}$, there exists $r>1$ such that $\omega\in RH_{r}$. Fix a constant $q_{r}$ such that $q_{r}>\max\{p,1\}$ and $({\frac{q_{r}}{p}})^{\prime}\leq r$.
\begin{theorem}
    If $0<p<\infty$ and $\omega\in A_{\infty}(\mathbb R^{n})$, then for any $\max\{q_{\omega},q_{r}\}<q<\infty$ and any $s\in\mathbb{Z}$ fulfilling $s\geq {\rm max}\{[n(\frac {q_\omega}{p}-1)],-1\}$, 
    \[
    h_{\omega}^{p}(\mathbb {R}^{n})=h_{\omega,atom}^{p,q,s}(\mathbb {R}^{n})
    \]
    with the equivalent norms.
\end{theorem}
In fact, Theorem 1.6 can be split into two parts as follows.
\begin{theorem}\label{th1.7}
    Let $0<p<\infty,\ \omega\in A_{\infty}(\mathbb R^{n}),\ q_\omega=\inf\{q\colon\omega\in A_{q}\},\ q_\omega<q<\infty$ and $s\in\mathbb{Z}$ fulfilling $s\geq {\rm max}\{[n(\frac {q_\omega}{p}-1)],-1\}.$ If $f\in h_\omega^{p}(\mathbb R^{n})\cap L^{q}(\mathbb R^{n})$, there exist a sequence of $\ \omega\text{-}(p,q,s)\text{-}$atoms $\{a_{j}\}_{j=1}^{\infty}$ with a corresponding sequence of non-negative numbers $\{\lambda_{j}\}_{j=1}^{\infty}$ and a sequence of $\omega\text{-}(p,q,s)\text{-}blocks$ $\{b_{j}\}_{j=1}^{\infty}$ with a corresponding sequence of non-negative numbers $\{\mu_{j}\}_{j=1}^{\infty}$ such that $f=\sum\limits_{j}\lambda_{j}a_{j}+\sum\limits_{j}\mu_{j}b_{j}$ and
\[
\left \|\left(\sum\limits_{j=1}^{\infty}\left(\frac{\lambda_j \chi_{Q_j}}{\omega(Q_j)^{\frac{1}{p}}}\right)^{\eta}\right)^{1/\eta}\right \|_{L_{\omega}^{p}}+\left \|\left(\sum\limits_{j=1}^{\infty}\left(\frac{\mu_j \chi_{P_j}}{\omega(P_j)^{\frac{1}{p}}}\right)^{\eta}\right)^{1/\eta}\right \|_{L_{\omega}^{p}}\leq C_{\eta}\left \|f\right \|_{h_{\omega}^{p}}
\]
for any $0<\eta<\infty$. Furthermore, the series converges to $f$ in both $h_{\omega}^{p}(\mathbb R^{n})$ and $L^{q}(\mathbb R^{n})$ norms.
\end{theorem}
\begin{theorem}\label{th1.8}
    Given $0<p<\infty,\ \omega\in A_{\infty}(\mathbb R^{n}),\ q_\omega=\inf\{q\colon\omega\in A_{q}\},\ q_{r}<q<\infty$ and $s\in\mathbb{Z}$ fulfilling $s\geq {\rm max}\{[n(\frac {q_\omega}{p}-1)],-1\}.$ Suppose that $\{a_{j}\}_{j=1}^{\infty}$ is a sequence of $\omega\text{-}(p,q,s)\text{-}$atoms with a corresponding sequence of non-negative numbers $\{\lambda_{j}\}_{j=1}^{\infty}$ and $\{b_{j}\}_{j=1}^{\infty}$ is a sequence of $\omega\text{-}(p,q,s)\text{-}$blocks with a corresponding sequence of non-negative numbers $\{\mu_{j}\}_{j=1}^{\infty}$ satisfying
\[
\left \|\sum\limits_{j=1}^{\infty}\frac{\lambda_j \chi_{Q_j}}{\omega(Q_j)^{\frac{1}{p}}}\right \|_{L_{\omega}^{p}}+\left \|\sum\limits_{j=1}^{\infty}\frac{\mu_j \chi_{P_j}}{\omega(P_j)^{\frac{1}{p}}}\right \|_{L_{\omega}^{p}}<\infty.
\]
Then the series $f=\sum\limits_{j}\lambda_{j}a_{j}+\sum\limits_{j}\mu_{j}b_{j}$ converges in $h_{\omega}^{p}(\mathbb{R}^{n})$ and satisfies 
\[
\left \|\sum\limits_{j}\lambda_{j}a_{j}\right \|_{h_{\omega}^{p}}\leq C\left \|\sum\limits_{j=1}^{\infty}\frac{\lambda_j \chi_{Q_j}}{\omega(Q_j)^{\frac{1}{p}}}\right \|_{L_{\omega}^{p}}
\]
and
\[
\left \|\sum\limits_{j}\mu_{j}b_{j}\right \|_{h_{\omega}^{p}}\leq C\left \|\sum\limits_{j=1}^{\infty}\frac{\mu_j \chi_{P_j}}{\omega(P_j)^{\frac{1}{p}}}\right \|_{L_{\omega}^{p}}.
\]
\end{theorem}
Theorem 1.6 follows from the atom decomposition theorem Theorem 1.7 and the reconstruction theorem Theorem 1.8 together with the fact $h^{p}_{\omega}(\mathbb R^{n})\cap L^{q}(\mathbb R^{n})$ is dense in $h^{p}_{\omega}(\mathbb R^{n})$. \par
As applications of the above atomic decomposition results, we shall prove the boundedness of the inhomogenous Calder\'on-Zygmund singular integrals and the local fractional integrals on weighted local Hardy spaces. The groundbreaking work of Hardy estimates for Calder\'on-Zygmund operators is completed by Stein and Weiss \cite{1960On}, Stein \cite{1961On}, and Fefferman and Stein \cite{fefferman1972}. In particular, weighted Hardy spaces estimates for singular integrals were proved by Str\"omberg and Torchinsky \cite{stromberg1989}. It is worth pointing out that the proof of Theorem~{\upshape\ref{th1.9}} and {\upshape\ref{th1.10}} is an adaption from the ones for local variable Hardy spaces in \cite{tan23}. Moreover, 
fractional integrals have been investigated extensively by several authors in recent years. 
Weighted Hardy space esitimates for fractional integrals were first proved by Str\"omberg and Wheeden \cite{10.2307/2000412}; see also Gatto et al. \cite{1985Fractional} and Tan \cite{tan21}. Theorem~{\upshape\ref{th1.12}} extends this result to weighted local Hardy spaces. We remark that the proof of this theorem is similar to the proof of \cite[Theorem~1.5]{2020A} but we need to concentrate on the differences.\par
Now we recall the inhomogeneous Calder\'on-Zygmund singular integrals in \cite{ding2021dual}. Define $\mathcal{D}(\mathbb R^{n})$ to be the space of all smooth functions with compact support. The operator $T$ is said to be an inhomogeneous Calder\'on-Zygmund integral if $T$ is a continuous linear operator from $\mathcal{D}$ to $\mathcal{D}^{\prime}$ defined by
\[
\left<T(f),g\right>=\int\mathcal{K}(x,y)f(y)g(x)dxdy
\]
for all $f,g\in\mathcal{D}(\mathbb R^{n})$ with disjoint supports, where $\mathcal{K}(x,y)$, the kernel of $T$, satisfies the conditions as follows.
\[
\vert\mathcal{K}(x,y)\vert\leq C\min\left\{\frac{1}{\vert x-y\vert^{n}},\frac{1}{\vert x-y\vert^{n+\delta}}\right\}\ {\rm for\ some\ \delta>0\ and}\ x\neq\ y
\]
and for $\epsilon\in(0,1)$
\[
\vert\mathcal{K}(x,y)-\mathcal{K}(x,y^{\prime})\vert+\vert\mathcal{K}(y,x)-\mathcal{K}(y^{\prime},x)\vert\leq C\frac{\vert y-y^{\prime}\vert^{\epsilon}}{\vert x-y\vert^{n+\epsilon}},
\]
where $\vert y-y^{\prime}\vert\leq\frac{1}{2}\vert x-y\vert$.
\begin{theorem}\label{th1.9}
    Let $0<p<\infty$ and $\omega\in A_{(\frac{n+\eta}{n})p}$ where $\eta=\epsilon\wedge\delta$. Suppose that $T$ is an inhomogeneous Calder\'on-Zygmund singular integral. If $T$ is a bounded operator on $L^{2}$, then $T$ can be extended to an $(h_{\omega}^{p}\text{-}L_{\omega}^{p})$ bounded operator. To be precise, there exists a constant $C$ such that
    \[
    \left \| T(f) \right \|_{L_{\omega}^{p}}\leq C\left \| f \right \|_{h_{\omega}^{p}}.
    \]
\end{theorem}
To state the $(h^{p}_{\omega}(\mathbb R^{n}),h^{p}_{\omega}(\mathbb R^{n}))$\text{-}boundedness of $T$, we assume one additional condition on $T$, $\int_{\mathbb R^{n}}T(a)(x)dx=0$ for the $\omega\text{-}(p,q,s)\text{-}$atom $a$. Then if $T$ satisfies the above moment condition, we write $T^{loc}_{*}(1)=0$.
\begin{theorem}\label{th1.10}
Let $0<p<\infty$ and $\omega\in A_{(\frac{n+\eta}{n})p}$ where $\eta=\epsilon\wedge\delta$. Suppose that $T$ is an inhomogeneous Calder\'on-Zygmund singular integral. If $T$ is a bounded operator on $L^{2}$ and $T^{loc}_{*}(1)=0$, then $T$ has a unique extension on $h^{p}_{\omega}$ and, moreover, there exists a constant $C$ such that
    \[
    \left \| T(f) \right \|_{h_{\omega}^{p}}\leq C\left \| f \right \|_{h_{\omega}^{p}}
    \]
    for all $f\in h_{\omega}^{p}$.
\end{theorem}
We also recall the following local fractional integral which is introduced by D. Yang and S. Yang \cite{2011Weighted}.
\begin{definition}
        Let $\alpha\in[0,n)$ and let $\varphi_{0}\in\mathcal{D}$ be such $\varphi_{0}\equiv1$ on $Q(0,1)$ and ${\rm supp}(\varphi_{0})\subset Q(0,2)$. The local fractional integral $I_{\alpha}^{loc}(f)$ of $f$ is defined by
    \[
    I_{\alpha}^{loc}(f)(x)\equiv\int_{\mathbb R^{n}}\frac{\varphi_{0}(y)}{\vert y\vert^{n-\alpha}}f(x-y)dy.
    \]
\end{definition}
Now we show that the local fractional integrals are bounded from $h^{p}(\omega^{p})$ to $L^{q}(\omega^{q})$ when $1<q<\infty$ and from $h^{p}(\omega^{p})$ to $h^{q}(\omega^{q})$ when $0<q\leq1$.
\begin{theorem}\label{th1.12}
Let $0<\alpha<n$ and $0<p<\frac{n}{\alpha}$. Define $q$ by $\frac{1}{q}=\frac{1}{p}-\frac{\alpha}{n}$. If a weight $\omega$ is such that $\omega^{p}\in RH_{\frac{q}{p}}$, then $I^{loc}_{\alpha}$ admits a bounded extension from $h^{p}(\omega^{p})$ to $L^{q}(\omega^{q})$ when $1<q<\infty$ and $I^{loc}_{\alpha}$ admits a bounded extension from $h^{p}(\omega^{p})$ to $h^{q}(\omega^{q})$ when $0<q\leq1$.
\end{theorem}

Throughout this paper, $C$ or $c$ denotes a positive constant that is independent of the main parameters involved but may vary at each occurrence. To denote the dependence of the constants on some parameter $s$, we will write $C_{s}$. We denote $f\leq Cg$ by $f\lesssim g$. If $f\lesssim g\lesssim f$, we write $f\sim g$ or $f\approx g$. Denote $Q(x,l(Q))$ the closed cube centered at $x$ and of side-length $l(Q)$. Similarly, given $Q=Q(x,l(Q))$ and $\lambda>0$, $\lambda Q$ means the cube with the same center $x$ and with side-length $\lambda l(Q)$. We denote $Q^{*}=2\sqrt{n}Q$. Moreover, we use the notation $j\wedge k=\min\{j,k\}$.

\section{Preliminaries}

In this section, we present some known results that will be used in the next sections and establish a new reproducing formula.\par
Firstly, we recall some known results about weights. For more details, see \cite{weights2011,2001fourier,1985Weighted}. Suppose that a weight $\omega$ is a non-negative, locally integrable function such that $0<\omega(x)<\infty$ for almost every $x\in\mathbb R^{n}$. It is said that $\omega$ is in the Muckenhoupt class $A_{p}$ for $1<p<\infty$ if
\[
[\omega]_{A_{p}}=\sup\limits_{Q}\left(\frac{1}{Q}\int_{Q}\omega(x)dx\right)\left(\frac{1}{Q}\int_{Q}\omega(x)^{-\frac{1}{p-1}}dx\right)^{p-1}<\infty,
\]
where $Q$ is any cube in $\mathbb {R}^{n}$ and when $p=1$, a weight $\omega\in A_{1}$ if for a.e. $x\in\mathbb {R}^{n}$,
\[
M\omega(x)\leq C\omega(x),
\]
where $M$ is the Hardy-Littlewood maximal operator defined by
\[
Mf(x)=\sup\limits_{x\in Q}\frac{1}{\vert Q\vert}\int_{Q}f(u)du.
\]
Therefore, define the set
\[
A_{\infty}(\mathbb {R}^{n})=\bigcup\limits_{1\leq p<\infty}A_{p}(\mathbb {R}^{n}).
\]
Given a weight $\omega\in A_{\infty}(\mathbb {R}^{n})$, define $$q_{\omega}=\inf\{q\geq1\colon \omega\in A_{q}\}.$$
Given a weight $\omega\in A_{\infty}(\mathbb {R}^{n})$ 
and $0<p<\infty$. Then the weighted Lebesgue space is defined by
\[
L_{\omega}^{p}(\mathbb {R}^{n})=\left\{f\colon\int\vert f(x)\vert^{p}\omega(x)dx<\infty\right\}
\]
where $f$ are measurable functions on $\mathbb {R}^{n}$. A weight $\omega\in A_{\infty}$ if and only if $\omega\in RH_{r}$ for some $r>1$: that is, for every cube $Q$,
\[
\left(\frac{1}{\vert Q\vert}\int_{Q}\omega(x)^{r}dx\right)^{\frac{1}{r}}\leq \frac{C}{\vert Q\vert}\int_{Q}\omega(x)dx.
\]
Furthermore, we can obtain the property that $\omega\in RH_{r}$ if and only if $\omega^{r}\in A_{\infty}$. Given $1<p,\ q<\infty$, a weight satisfies the $A_{p,q}$ condition of Muckenhoupt and Wheeden if for every cube $Q$,
\[
\left(\frac{1}{\vert Q\vert}\int_{Q}\omega^{q}dx\right)^{\frac{1}{q}}\left(\frac{1}{\vert Q\vert}\int_{Q}\omega^{-p^{\prime}}dx\right)^{\frac{1}{p^{\prime}}}\leq C.
\]
It follows from the defintion that $\omega\in A_{p,q}$ if and only if $\omega^{q}\in A_{1+\frac{q}{p^{\prime}}}$. When $p=1$ and $q>1$, it is said that $\omega\in A_{1,q}$ if for every cube $Q$ and almost every $x\in Q$,
\[
\frac{1}{\vert Q\vert}\int_{Q}\omega(x)^{q}dx\leq C\omega(x)^{q},
\]
which is clearly equivalent to $\omega^{q}\in A_{1}$.\\
Given $0\leq\alpha<n$ and $1<p<\frac{n}{\alpha}$, define $q$ by $\frac{1}{p}-\frac{1}{q}=\frac{\alpha}{n}$. If $\omega\in A_{p,q}$, the fractional maximal operator
\[
M_{\alpha}(f)(x)=\sup\limits_{Q}\vert Q\vert^{\frac{\alpha}{n}}\left(\frac{1}{\vert Q\vert}\int_{Q}\vert f(y)\vert dy\right)\chi_{Q}(x)
\]
is bounded from $L^{p}(\omega^{p})$ to $L^{q}(\omega^{q})$.\par
Now we recall two lemmas which will be applied to the proofs in Section~{\upshape\ref{sec3}}.
First we need the weighted Fefferman--Stein vector-valued maximal inequality \cite{andersen1981weighted} as follows.
\begin{lemma}\label{le2.1}
    Let $1<p,\ q<\infty$, $\omega\in A_{p}(\mathbb {R}^{n}),\ f=\{f_{i}\}_{i\in\mathbb{Z}},\ f_{i}\in L_{loc},$\\
\[
\left \| \left \|\mathbb{M}(f)\right \|_{l^{q}} \right \|_{L_{\omega}^{p}}\leq C\left \| \left \|f\right \|_{l^{q}} \right \|_{L_{\omega}^{p}}
\]
where $\mathbb{M}(f)=\{M(f_{i})\}_{i\in\mathbb{Z}}.$
\end{lemma}
\begin{remark}\label{re2.2}
    If we let $f_{i}=\chi_{Q_{i}}$, for some collection of cubes $Q_{i}$, then given $0<p<\infty$, $\tau>1$ and $\omega\in A_{\infty}$, there exists $r>1$ such that $\omega\in A_{rp}$. Thus we have that
    \[
    \left \| \sum\limits_{i}\chi_{\tau Q_{i}} \right \|_{L^{p}(\omega)}\lesssim\left \| \left(\sum\limits_{i}(M\chi_{Q_{i}})^{r}\right)^{\frac{1}{r}} \right \|^r_{L^{rp}(\omega)}\lesssim\left \| \sum\limits_{i}\chi_{Q_{i}} \right \|_{L^{p}(\omega)}.
    \]
\end{remark}
\begin{lemma}[\cite{2020A}]\label{le2.3}
    Fix $q>1$. Suppose that $0<p<q$ and $\omega\in RH_{({\frac{q}{p}})^{\prime}}$. We are given countable collections of cubes $\{Q_{j}\}_{j=1}^{\infty}$, of non-negative numbers $\{\lambda_{j}\}_{j=1}^{\infty}$ and of non-negative measurable functions $\{a_{j}\}_{j=1}^{\infty}$ such that ${\rm supp}(a_{j})\subset Q_{j}$, $\left \|a_{j}\right \|_{L^{q}}\leq \vert Q_{j}\vert^{\frac{1}{q}}\omega(Q_{j})^{-\frac{1}{p}}$. Then
\[
\left \|\sum\limits_{j=1}^{\infty}\lambda_{j}a_{j}\right\|_{L_{\omega}^{p}}\leq C\left \|\sum\limits_{j=1}^{\infty}\frac{\lambda_{j}\chi_{Q_{j}}}{\omega(Q_{j})^{\frac{1}{p}}}\right \|_{L_{\omega}^{p}}.
\]
\end{lemma}
In order to obtain the atomic decomposition, we need a new reproducing formula. Thus we introduce test functions as follows.
\begin{definition}
    Let $\psi_0,\ \psi\in \mathcal{S}(\mathbb R^{n})$ satisfies
\begin{equation}
{\rm supp}\psi_0\subseteq\{x\in\mathbb R^{n}\colon\vert x\vert\leq1\};\ \int{\psi_0}=1,
\end{equation}
\begin{equation}
{\rm supp}\psi\subseteq\{x\in\mathbb R^{n}\colon\vert x\vert\leq1\};\ \int{\psi(x)}x^{\alpha}dx=0,\ {\rm for\ all}\ \vert\alpha\vert\leq M,
\end{equation}
and
\begin{equation}
\vert\widehat{\psi_0}(\xi)\vert^{2}+ \sum\limits_{j=1}^{\infty}\vert\widehat{\psi}(2^{-j}\xi)\vert^{2}=1,\ {\rm for\ all}\ \xi\in\mathbb R^{n},
\end{equation}
where a constant $M=M_{p,n}$ is large enough.
\end{definition}
\begin{lemma}\label{le2.5}
    Let $0<p<\infty,\ \omega\in A_{\infty}(\mathbb R^{n}),\ q_\omega=\inf\{q\colon\omega\in A_{q}\}$ and $q_\omega<q<\infty.$ Suppose that $\psi_0,\ \psi\in \mathcal{S}(\mathbb R^{n})$ satisfies (2.1)-(2.3). Then there exists a positive integer $N$ such that for any $f\in h_\omega^{p}(\mathbb R^{n})\cap L^{q}(\mathbb R^{n})$
\[
f(x)=\sum\limits_{j\in\mathbb N}\sum\limits_{Q\in\Pi_{j+N}}\vert Q\vert\psi_{j}(x-u_{Q})(\psi_{j}\ast h)(u_{Q})
\]
where $u_Q$ is any point in $Q$ and $h\in h_\omega^{p}(\mathbb R^{n})\cap L^{q}(\mathbb R^{n})$ satisfies
\[
\left \|h\right \|_{L^{q}(\mathbb R^{n})}\sim\left \|f\right \|_{L^{q}(\mathbb R^{n})},\ \left \|h\right \|_{h_{\omega}^{p}(\mathbb R^{n})}\sim\left \|f\right \|_{h_{\omega}^{p}(\mathbb R^{n})}.
\]
Moreover, the series converges in $L^{q}(\mathbb {R}^{n})$.
\end{lemma}
\begin{proof}
Applying the Calder\'on reproducing formula on  $L^{2}$ and the Coifman's decomposition, we have that
\[
\begin{aligned}
    f(x)
    &=\sum\limits_{j\in\mathbb N}\psi_{j}\ast\psi_{j}\ast f(x)\\
    &=\sum\limits_{j\in\mathbb N}\sum\limits_{Q\in\Pi_{j+N}}\int_{Q}\psi_{j}(x-u)(\psi_{j}\ast f)(u)du\\
    &=:T_{N}(f)(x)+R_{n}(f)(x)
\end{aligned}
\]
where
\[
T_{N}(f)(x)=\sum\limits_{j\in\mathbb N}\sum\limits_{Q\in\Pi_{j+N}}\vert Q\vert\psi_{j}(x-u_{Q})(\psi_{j}\ast f)(u_{Q}),
\]
\[
R_{N}(f)(x)=\sum\limits_{j\in\mathbb N}\sum\limits_{Q\in\Pi_{j+N}}\int_{Q}[\psi_{j}(x-u)(\psi_{j}\ast f)(u)-\psi_{j}(x-u_{Q})(\psi_{j}\ast f)(u_{Q})]du,
\]
where some larger integer $N$ will be chosen later and $u_{Q}$ is any point in $Q$.\par
Details are similar to those in \cite{han1994calderon,Li2021,wu2012atomic}. 
By a standard almost orthogonality estimation, we can prove that
\[
\left \|R_{N}(f)\right \|_{h_{\omega}^{p}(\mathbb R^{n})}\leq C2^{-N}\left \|f\right \|_{h_{\omega}^{p}(\mathbb R^{n})}
\]
and
\[
\left \|R_{N}(f)\right \|_{L^{q}(\mathbb R^{n})}\leq C2^{-N}\left \|f\right \|_{L^{q}(\mathbb R^{n})}.
\]\par
We can choose $N$ large enough so that $C2^{-N}<1$. Since $I=T_{N}+R_{N}$ and $R_{N}$ is bounded on $h_{\omega}^{p}(\mathbb R^{n})$ and $L^{q}(\mathbb R^{n})$, then $T_{N}$ and $T_{N}^{-1}$ are bounded on $h_{\omega}^{p}(\mathbb R^{n})$ and $L^{2}(\mathbb R^{n})$. Moreover, $T_{N}^{-1}=\sum\limits_{n=0}^{\infty}(R_{N})^{n}$. Let $h(x)=T_{N}^{-1}(f)(x)$ and then
\[
\left \|h\right \|_{h_{\omega}^{p}(\mathbb R^{n})}\sim\left \|f\right \|_{h_{\omega}^{p}(\mathbb R^{n})},\ \left \|h\right \|_{L^{q}(\mathbb R^{n})}\sim\left \|f\right \|_{L^{q}(\mathbb R^{n})}.
\]
Furthermore,
\[
f(x)=T_{N}(T_{N}^{-1}(f))(x)=\sum\limits_{j\in\mathbb N}\sum\limits_{Q\in\Pi_{j+N}}\vert Q\vert\psi_{j}(x-u_{Q})(\psi_{j}\ast h)(u_{Q}).
\]
where the series converges in $L^{2}$.\par
Next we will prove that the series above converges in $L^{q}$ for any $1<q<\infty$. Since $L^{q}\cap L^{2}$ is dense in $L^{q}$, it suffices to show that the series converges in $L^{q}$ for any function $f\in L^{q}\cap L^{2}$.\\Let
\[
B_{l}=\{Q\colon l(Q)=2^{-j-N},\ Q\subset B(0,l),\ \vert j\vert\leq l\},
\]
where $B(0,l)$ are balls centered at origin with radii $l$ in $\mathbb {R}^{n}$. Write $\psi_{Q}=\psi_{j}$. We claim that for each function $f\in L^{q}\cap L^{2}$
\[
\left \|\sum\limits_{l>L}\sum\limits_{Q\in B_{l}}\vert Q\vert\psi_{Q}(x-u_{Q})(\psi_{Q}\ast h)(u_{Q})\right \|_{L^{q}}\to 0,\ as\ L\to+\infty.
\]
In fact, by duality argument, we have that
\[
\begin{aligned}
        &\left \|\sum\limits_{l>L}\sum\limits_{Q\in B_{l}}\vert Q\vert\psi_{Q}(x-u_{Q})(\psi_{Q}\ast h)(u_{Q})\right \|_{L^{q}}\\
        &=\sup\limits_{\left \|g\right \|_{L^{q^{\prime}}}\leq1}\left \langle \sum\limits_{l>L}\sum\limits_{Q\in B_{l}}\vert Q\vert\psi_{Q}(x-u_{Q})(\psi_{Q}\ast h)(u_{Q}),g \right \rangle\\
        &=\sup\limits_{\left \|g\right \|_{L^{q^{'}}}\leq1}\left\vert\sum\limits_{l>L}\sum\limits_{Q\in B_{l}}\vert Q\vert\psi_{Q}\ast h(u_{Q})\psi_{Q}\ast g(u_{Q})\right\vert\\
        &\leq\sup\limits_{\left \|g\right \|_{L^{q^{'}}}\leq1}\left\vert\int_{\mathbb R^{n}}\sum\limits_{l>L}\sum\limits_{Q\in B_{l}}(\psi_{Q}\ast h)(u_Q)(\psi_{Q}\ast g)(u_{Q})\chi_{Q}(y)dy\right\vert\\
        &\leq\sup\limits_{\left \|g\right \|_{L^{q^{\prime}}}\leq1}\int_{\mathbb R^{n}}\left\{\sum\limits_{l>L}\sum\limits_{Q\in B_{l}}\vert(\psi_{Q}\ast h)(u_Q)\vert^2\chi_{Q}(y)\right\}^{\frac{1}{2}}\\
        &\quad\times\left\{\sum\limits_{l>L}\sum\limits_{Q\in B_{l}}\vert(\psi_{Q}\ast g)(u_Q)\vert^2\chi_{Q}(y)\right\}^{\frac{1}{2}}dy\\
        &\leq\sup\limits_{\left \|g\right \|_{L^{q^{\prime}}}\leq1}\left \|\left\{\sum\limits_{l>L}\sum\limits_{Q\in B_{l}}\vert(\psi_{Q}\ast g)(u_Q)\vert^2\chi_{Q}(y)\right\}^{\frac{1}{2}}\right \|_{L^{q^{\prime}}}\\
        &\quad\times\left \|\left\{\sum\limits_{l>L}\sum\limits_{Q\in B_{l}}\vert(\psi_{Q}\ast h)(u_Q)\vert^2\chi_{Q}(y)\right\}^{\frac{1}{2}}\right \|_{L^{q}}\\
        &\leq C\left \|\left\{\sum\limits_{l>L}\sum\limits_{Q\in B_{l}}\vert(\psi_{Q}\ast h)(u_Q)\vert^2\chi_{Q}(y)\right\}^{\frac{1}{2}}\right \|_{L^{q}}
\end{aligned}
\]
which tends to zero as $L$ goes to infinity. Then by a standard density argument, we can obtain the desired result.
\end{proof}
By Lemma~{\upshape\ref{le2.5}}, we can obtain the following corollary.
\begin{corollary}\label{co2.6}
Let $0<p<\infty,\ \omega\in A_{\infty}(\mathbb R^{n}),\ q_\omega<q<\infty.$ Suppose that $\psi_0,\ \psi\in \mathcal{S}(\mathbb R^{n})$ satisfies (2.1)-(2.3). Then for any $f\in h_\omega^{p}(\mathbb R^{n})\cap L^{q}(\mathbb R^{n}),$
\[
\left \|f\right \|_{h_{\omega}^{p}}\sim\left \|\left\{\sum\limits_{j\in\mathbb N}\sum\limits_{Q\in\Pi_{j+N}}\sup_{u\in Q}\vert\psi_{j}\ast f(u)\vert^2\chi_{Q}(x)\right\}^{1/2}\right \|_{L_{\omega}^{p}}
\]
\end{corollary}
\begin{proof}
     From the above proof, we know that 
\[
\begin{aligned}
    \left \|T_{N}(f)\right \|_{h_{\omega}^{p}}
    &=\left \|\sum\limits_{j\in\mathbb N}\sum\limits_{Q\in\Pi_{j+N}}\vert Q\vert\psi_{j}(x-u_{Q})(\psi_{j}\ast f)(u_{Q})\right \|_{h_{\omega}^{p}}\\
    &\leq C\left \|\left\{\sum\limits_{j\in\mathbb N}\sum\limits_{Q\in\Pi_{j+N}}\vert(\psi_{j}\ast f)(u_{Q})\vert^{2}\chi_{Q}\right\}^{\frac{1}{2}}\right \|_{L_{\omega}^{p}}.
\end{aligned}
\]
Hence, for any $f\in L_\omega^{q}(\mathbb R^{n})\cap h_\omega^{p}(\mathbb R^{n})$, we can obtain that
\[
\begin{aligned}
\left \|f\right \|_{h_{\omega}^{p}}
&=\left \|T_{N}^{-1}\circ T_{N}(f)\right \|_{h_{\omega}^{p}}\\
&\leq C\left \|T_{N}(f)\right \|_{h_{\omega}^{p}}\\
&\leq C\left \|\left\{\sum\limits_{j\in\mathbb N}\sum\limits_{Q\in\Pi_{j+N}}\vert(\psi_{j}\ast f)(u_{Q})\vert^{2}\chi_{Q}\right\}^{\frac{1}{2}}\right \|_{L_{\omega}^{p}}
\end{aligned}
\]
which implies that
\[
\left \|f\right \|_{h_{\omega}^{p}}\leq C\left \|\left\{\sum\limits_{j\in\mathbb N}\sum\limits_{Q\in\Pi_{j+N}}\inf\limits_{u\in Q}\vert(\psi_{j}\ast f)(u)\vert^{2}\chi_{Q}\right\}^{\frac{1}{2}}\right \|_{L_{\omega}^{p}}
\]
Then, repeating the same process, we can obtain that
\[
\left \|\left\{\sum\limits_{j\in\mathbb N}\sum\limits_{Q\in\Pi_{j+N}}\sup\limits_{u\in Q}\vert(\psi_{j}\ast f)(u)\vert^{2}\chi_{Q}\right\}^{\frac{1}{2}}\right \|_{L_{\omega}^{p}}\leq C\left \|f\right \|_{h_{\omega}^{p}}
\]
Details are similar to those in \cite{ding2019discrete}. Furthermore, we have that
\[
\begin{aligned}
    \left \|f\right \|_{h_{\omega}^{p}}
    &\approx\left \|\left\{\sum\limits_{j\in\mathbb N}\sum\limits_{Q\in\Pi_{j+N}}\sup\limits_{u\in Q}\vert(\psi_{j}\ast f)(u)\vert^{2}\chi_{Q}\right\}^{\frac{1}{2}}\right \|_{L_{\omega}^{p}}\\
\end{aligned}
\]
\end{proof}
Then we give the following lemma which is need for the proof of Theorem~{\upshape\ref{th1.8}}. The proof of the lemma is similar to but easier than those in \cite{ding2012boundedness,tan2019atomic}.\par
\begin{lemma}\label{le2.7}
    Let $0<p<\infty$, $\omega\in A_{\infty}(\mathbb {R}^{n})$. Then for any $f\in \mathcal{S}^{\prime}(\mathbb {R}^{n})$
    \[
    \left \| f \right \|_{h_{\omega}^{p}}\sim \left \| g(f) \right \|_{L_{\omega}^{p}}.
    \]
\end{lemma}\par
We also recall the following key lemmas which are need for the proof of Section~{\upshape\ref{sec4}}. For more details, see \cite{2020A}.
\begin{lemma}\label{le2.8}
Fix $q>1$. If $0<p<q$ and $\omega\in RH_{(\frac{q}{p})^{\prime}}$, then for all sequences of cubes $\{Q_{k}\}$ and non-negative functions $\{g_{k}\}$ such that ${\rm supp}(g_{k})\subset Q_{k}$,
\[
\left \| \sum\limits_{k}g_{k} \right \|_{L^{p}(\omega)}\lesssim\left \| \sum\limits_{k}\left(\frac{1}{\vert Q_{k}\vert}\int_{Q_{k}}g_{k}^{q}dy\right)^{\frac{1}{q}}\chi_{Q_{k}} \right \|_{L^{p}(\omega)}.
\]
\end{lemma}
\begin{lemma}\label{le2.9}
    Suppose $0<\alpha<n$, $0<p<\frac{n}{\alpha}$, and $\frac{1}{q}=\frac{1}{p}-\frac{\alpha}{n}$. If $\omega^{p}\in RH_{\frac{q}{p}}$, then for any countable collection of cubes $\{Q_{k}\}$ and $\lambda_{k}>0$,
    \[
    \left \| \sum\limits_{k}\lambda_{k}\vert Q_{k}\vert^{\frac{\alpha}{n}}\chi_{Q_{k}} \right \|_{L^{q}(\omega^{q})}\lesssim\left \| \sum\limits_{k}\lambda_{k}Q_{k} \right \|_{L^{p}(\omega^{p})}.
    \]
\end{lemma}
\begin{lemma}\label{le2.10}
    Fix $N\geq0$ and $0\leq\alpha<n$. Let $\mathcal{K}$ be a distribution such that $\vert\mathcal{\widehat K}(\xi)\vert\lesssim\vert\xi\vert^{-\alpha}$. Define the operator $T$ by $Tf=\mathcal{K}\ast f$. Let $a$ be any $(p,q,N)$\text{-}atom or $(p,q,N)$\text{-}block with {\rm supp}$(a)\subset Q$ for $0<p<\infty$ and $1\leq q<\infty$. Then for all $x\in(Q^{*})^{c}$,
    \[
    M_{\Phi}(Ta)(x)\lesssim M_{\alpha_{\tau}}(\chi_{Q})(x)^{\tau},
    \]
    where $\tau=\frac{n+N+1}{n}$ and $\alpha_{\tau}=\alpha/\tau$.
\end{lemma}
\begin{lemma}\label{le2.11}
    Given $0<\alpha<n$, $1<r<\infty$, and $1<p<\frac{n}{\alpha}$, define $q$ by $\frac{1}{p}-\frac{1}{q}=\frac{\alpha}{n}$. If $\omega\in A_{p,q}$, then
    \[
        \left \| \left(\sum\limits_{k}(M_{\alpha}g_{k})^{r}\right)^{\frac{1}{r}} \right \|_{L^{q}(\omega^{q})}\lesssim\left \| \left(\sum\limits_{k}\vert g_{k}\vert^{r}\right)^{\frac{1}{r}} \right \|_{L^{p}(\omega^{p})}.
    \]
\end{lemma}
\section{Proofs of Theorems 1.7 and 1.8}\label{sec3}
In this section, we will establish the atomic decomposition characterization of $h_{\omega}^{p}(\mathbb {R}^{n})$ for $0<p<\infty$ and $\omega\in A_{\infty}$. Now we give the proof of the atom decomposition.
\begin{proof}[Proof of Theorem~{\upshape\ref{th1.7}}]
Suppose that $f\in h_\omega^{p}\cap L^{q}$, $0<p<\infty,\ q_\omega<q<\infty.$ By Lemma~{\upshape\ref{le2.5}}, we can obtain
\[
\begin{aligned}
f(x)
&=\sum\limits_{j\in\mathbb N}\sum\limits_{Q\in\Pi_{j+N}}\vert Q\vert\psi_{j}(x-u_{Q})(\psi_{j}\ast h)(u_{Q})\\
&=\sum\limits_{Q\in\Pi_{N}}\vert Q\vert\psi_{0}(x-u_{Q})(\psi_{0}\ast h)(u_{Q})+\sum\limits_{j\geq1}\sum\limits_{Q\in\Pi_{j+N}}\vert Q\vert\psi_{j}(x-u_{Q})(\psi_{j}\ast h)(u_{Q})\\
&=\uppercase\expandafter{\romannumeral1}+\uppercase\expandafter{\romannumeral2}.
\end{aligned}
\]
Define
\[
S^{0}(h)(x)=\left\{\sum\limits_{P\in\Pi_{N}}\sup\limits_{u\in P}\vert\psi_{0}\ast h(u)\vert^2\chi_{P}(x)\right\}^{1/2}
\]
and
\[
S^{1}(h)(x)=\left\{\sum\limits_{j\geq1}\sum\limits_{Q\in\Pi_{j+N}}\sup\limits_{u\in Q}\vert\psi_{j}\ast h(u)\vert^2\chi_{Q}(x)\right\}^{1/2}.
\]
For any $ i\in\mathbb{Z}$ and $k=0,1$, set
\[
\Omega_{i,k}=\left\{x\in\mathbb{R}^{n}\colon S^{k}(h)(x)>2^i\right\}
\]
and
\[
\widetilde\Omega_{i,k}=\left\{x\in\mathbb{R}^{n}\colon M(\chi_{\Omega_{i,k}})(x)>\frac{1}{10^{n}}\right\}.
\]
Denote
\[
B_{i,0}=\left\{P\colon P\in\Pi_{N},\vert P\cap\Omega_{i,0}\vert>\frac{1}{2}\vert P\vert,\vert P\cap\Omega_{i+1,0}\vert\leq\frac{1}{2}\vert P\vert\right\}
\]
and
\[
B_{i,1}=\left\{Q\colon Q\in\bigcup_{j\geq1}\Pi_{j+N},\vert Q\cap\Omega_{i,1}\vert>\frac{1}{2}\vert Q\vert,\vert Q\cap\Omega_{i+1,1}\vert\leq\frac{1}{2}\vert Q\vert\right\}.
\]
Denote that $\widetilde{Q}\in B_{i,1}$ are maximal dyadic cubes in $B_{i,1}$. If$\ l(Q)=2^{-j-N}$, use $\psi_{Q}$ to denote $\psi_{j}$.\par
Now we estimate $\uppercase\expandafter{\romannumeral2}$. We can rewrite
\[
\begin{aligned}
\uppercase\expandafter{\romannumeral2}
&=\sum_{i=-\infty}^{+\infty}\sum_{\widetilde{Q}\in B_{i,1}}\sum_{Q\subset\widetilde{Q},Q\in B_{i,1}}\vert Q\vert(\psi_{Q}\ast h)(u_Q)\psi_{Q}(x-u_Q)\\
&=:\sum_{i=-\infty}^{+\infty}\sum_{\widetilde{Q}\in B_{i,1}}\lambda_{\widetilde{Q}}^{i}a_{\widetilde{Q}}^{i}(x),
\end{aligned}
\]
where
\[
a_{\widetilde{Q}}^{i}(x):=\frac{1}{\lambda_{\widetilde{Q}}^{i}}\sum_{Q\subset\widetilde{Q}}\vert Q\vert(\psi_{Q}\ast h)(u_Q)\psi_{Q}(x-u_Q)
\]
and
\[
\lambda_{\widetilde{Q}}^{i}:=\widetilde C\frac{\omega(\widetilde Q)^{\frac{1}{p}}}{\vert\widetilde Q\vert^{\frac{1}{q}}}\left \|\left\{\sum_{Q\subset\widetilde{Q}}\vert\psi_{Q}\ast h(u_{Q})\vert^2\chi_{Q}\right\}^{1/2}\right \|_{L^{q}}.
\]
By the definition of $\psi_{Q}$, we find that $a_{\widetilde{Q}}^{i}$ is supported in $c_{1}\widetilde Q$ where $c_{1}=2^{N+3}$ and the vanishing moment condition of $a_{\widetilde{Q}}^{i}$ follows from the vanishing moment condition of $\psi_{Q}$. There exists a constant $C\geq1$ such that $\vert c_{1}\widetilde Q\vert\leq C$. Then we try to obtain the size condition of $a_{\widetilde{Q}}^{i}$. By the duality argument,
\[
\begin{aligned}
 &\left \|\sum_{Q\subset\widetilde{Q}}\vert Q\vert(\psi_{Q}\ast h)(u_Q)\psi_{Q}(x-u_Q)\right \|_{L^{q}}\\
 &=\sup\limits_{\left \|g\right \|_{L^{q^{\prime}}}\leq1}\left \langle \sum_{Q\subset\widetilde{Q}}\vert Q\vert(\psi_{Q}\ast h)(u_Q)\psi_{Q}(x-u_Q),g \right \rangle\\
 &=\sup\limits_{\left \|g\right \|_{L^{q^{\prime}}}\leq1}\left\vert\int_{\mathbb R^{n}}\sum_{Q\subset\widetilde{Q}}(\psi_{Q}\ast h)(u_Q)(\psi_{Q}\ast g)(u_{Q})\chi_{Q}(y)dy\right\vert\\
 &\leq\sup\limits_{\left \|g\right \|_{L^{q^{\prime}}}\leq1}\int_{\mathbb {R}^{n}}\left(\sum_{Q\subset\widetilde{Q}}\vert(\psi_{Q}\ast h)(u_Q)\vert^2\chi_{Q}(y)\right)^{\frac{1}{2}}\\
 &\quad\times\left(\sum_{Q\subset\widetilde{Q}}\vert(\psi_{Q}\ast g)(u_Q)\vert^2\chi_{Q}(y)\right)^{\frac{1}{2}}dy\\
 &\leq\sup\limits_{\left \|g\right \|_{L^{q^{\prime}}}\leq1}\left\{\int_{\mathbb R^{n}}\left(\sum_{Q\subset\widetilde{Q}}\vert(\psi_{Q}\ast h)(u_Q)\vert^2\chi_{Q}(y)\right)^{\frac{q}{2}}dy\right\}^{\frac{1}{q}}\\
 &\quad\times\left\{\int_{\mathbb R^{n}}\left(\sum_{Q\subset\widetilde{Q}}\vert(\psi_{Q}\ast g)(u_Q)\vert^2\chi_{Q}(y)\right)^{\frac{q^{\prime}}{2}}dy\right\}^{\frac{1}{q^{\prime}}}\\
 &\leq\sup\limits_{\left \|g\right \|_{L^{q^{\prime}}}\leq1}\left \|S^{1}(g)\right \|_{L^{q^{\prime}}}\left \|\left\{\sum_{Q\subset\widetilde{Q}}\vert(\psi_{Q}\ast h)(u_Q)\vert^2\chi_{Q}(y)\right\}^{\frac{1}{2}}\right \|_{L^{q}}.
\end{aligned}
\]
Therefore we can choose an appropriate constant $\widetilde C$ such that
\[
\left \|a_{\widetilde{Q}}^{i}(x)\right \|_{L^{q}}\leq\frac{\vert\widetilde Q\vert^{\frac{1}{q}}}{\omega(\widetilde Q)^{\frac{1}{p}}}.
\]
In conclusion, each $a_{\widetilde{Q}}^{i}(x)$ is a $\omega\text{-}(p,q,s)\text{-}$atom of $h_\omega^{p}(\mathbb R^{n})$.\par
Then, we try to prove that for any $0<\eta<\infty$, we have
\[
\left \|\left(\sum\limits_{i}\sum\limits_{\widetilde Q\in B_{i,1}}\left(\frac{\lambda_{\widetilde{Q}}^{i}\chi_{c_{1}\widetilde Q_i}}{\omega(\widetilde Q)^{\frac{1}{p}}}\right)^{\eta}\right)^{1/\eta}\right \|_{L_{\omega}^{p}}\leq C\left \|f\right \|_{h_{\omega}^{p}}.
\]
Since $\bigcup\limits_{Q\in B_{i,1}}Q\subseteq\widetilde\Omega_{i,1}$, note that $\widetilde Q_{i}\subset\widetilde\Omega_{i,1}$ when $\widetilde Q\in B_{i,1}$.\\
We claim that
\begin{equation}\label{cl3.1}
    \left \|\left\{\sum_{Q\subset\widetilde{Q}}\vert\psi_{Q}\ast h(u_{Q})\vert^2\chi_{Q}\right\}^{\frac{1}{2}} \right\|_{L^{q}}\leq C2^{i}\vert\widetilde Q\vert^{\frac{1}{q}}.
\end{equation}
When $x\in Q$ and $Q\in B_{i,1}$, $M(\chi_{Q\cap\widetilde\Omega_{i,1}\backslash\Omega_{i+1,1}})(x)>\frac{1}{2}.$ Moreover, since
\[
\chi_{Q}(x)\leq2M(\chi_{Q\cap\widetilde\Omega_{i,1}\backslash\Omega_{i+1,1}})(x),
\]
then
\[
\chi_{Q}(x)\leq4M^{2}(\chi_{Q\cap\widetilde\Omega_{i,1}\backslash\Omega_{i+1,1}})(x).
\]
By Lemma~{\upshape\ref{le2.1}}, for any $1<q<\infty$
\[
\begin{aligned}
&\left \|\left\{\sum_{Q\subset\widetilde{Q}}\vert\psi_{Q}\ast h(u_{Q})\vert^2\chi_{Q}\right\}^{1/2}\right \|_{L^{q}}^{q}\\
&=\int_{\mathbb R^{n}}\left(\sum_{Q\subset\widetilde{Q}}\vert\psi_{Q}\ast h(u_{Q})\vert^2\chi_{Q}(x)\right)^{\frac{q}{2}}dx\\
&\leq C\int_{\mathbb R^{n}}\left(\sum_{Q\subset\widetilde{Q}}\vert\psi_{Q}\ast h(u_{Q})\vert^2 M^{2}(\chi_{Q\cap\widetilde\Omega_{i,1}\backslash\Omega_{i+1,1}})(x)\right)^{\frac{q}{2}}dx\\
&\leq C\int_{\mathbb R^{n}}\left(\sum_{Q\subset\widetilde{Q}}\vert\psi_{Q}\ast h(u_{Q})\vert^2 \chi_{Q\cap\widetilde\Omega_{i,1}\backslash\Omega_{i+1,1}}(x)\right)^{\frac{q}{2}}dx\\
&\leq C\int_{\widetilde Q\cap\widetilde\Omega_{i,1}\backslash\Omega_{i+1,1}}\left(\sum_{Q\subset\widetilde{Q}}\vert\psi_{Q}\ast h(u_{Q})\vert^2 \chi_{Q}(x)\right)^{\frac{q}{2}}dx\\
&\leq C\int_{\widetilde Q\cap\widetilde\Omega_{i}\backslash\Omega_{i+1}}\left(S^{1}(h)\right)^{q}dx\\
&\leq C2^{iq}\vert\widetilde Q\vert.
\end{aligned}
\]
Hence we finished the proof of the claim ({\upshape\ref{cl3.1}}). Now we can obtain
\[
\left \|\left(\sum\limits_{i}\sum\limits_{\widetilde Q\in B_{i,1}}\left(\frac{\lambda_{\widetilde{Q}}^{i}\chi_{c_{1}\widetilde Q_i}}{\omega(\widetilde Q)^{\frac{1}{p}}}\right)^{\eta}\right)^{1/\eta}\right \|_{L_{\omega}^{p}}\leq C\left \|\left(\sum\limits_{i}\sum\limits_{\widetilde Q\in B_{i,1}}\left(2^{i}\chi_{c_{1}\widetilde Q_{i}}\right)^{\eta}\right)^{1/\eta}\right \|_{L_{\omega}^{p}}.
\]
Since $\Omega_{i,1}\subset\widetilde\Omega_{i,1}$ for any $i\in\mathbb{Z}$ and $\vert\widetilde\Omega_{i,1}\vert\leq C\vert\Omega_{i,1}\vert$ for any $x\in\mathbb{R}^{n}$, we have
\[
\chi_{\widetilde\Omega_{i,1}}(x)\leq CM^{\gamma}(\chi_{\Omega_{i,1}})(x)
\]
where $\gamma$ is large enough such that $\gamma p>q_{\omega}$ and $\gamma \eta>1$. Applying Lemma~{\upshape\ref{le2.1}} with $\omega\in A_{\gamma p}$, we can obtain
\[
\begin{aligned}
&\left \|\left(\sum\limits_{i}\sum\limits_{\widetilde Q\in B_{i,1}}\left(2^{i}\chi_{c_{1}\widetilde Q_{i}}\right)^{\eta}\right)^{1/\eta}\right \|_{L_{\omega}^{p}}\leq C\left \|\left(\sum\limits_{i}\left(2^{i}\chi_{\widetilde \Omega_{i,1}}\right)^{\eta}\right)^{\frac{1}{\eta}}\right \|_{L_{\omega}^{p}}\\
&\leq C\left \|\left(\sum\limits_{i}\left(2^{i/\gamma}M(\chi_{\Omega_{i,1}})\right)^{\gamma\eta}\right)^{1/\eta}\right \|_{L_{\omega}^{p}}=C\left \|\left(\sum\limits_{i}\left(2^{i/\gamma}M(\chi_{ \Omega_{i,1}})\right)^{\gamma\eta}\right)^{1/\gamma\eta}\right \|_{L_{\omega}^{\gamma p}}^{\gamma}\\
&\leq C\left \|\left(\sum\limits_{i}2^{i\eta}\chi_{ \Omega_{i,1}}\right)^{1/\eta}\right \|_{L_{\omega}^{p}}.
\end{aligned}
\]
It's easy to know that $\Omega_{i+1,1}\subset\Omega_{i,1}$ and $\vert\bigcap\limits_{i=1}^{\infty}\Omega_{i,1}\vert=0$. Then for a.e. $x\in\mathbb{R}^{n}$, we have
\[
\left(\sum\limits_{i}2^{i\eta}\chi_{ \Omega_{i,1}}(x)\right)^{\frac{1}{\eta}}\sim\left(\sum\limits_{i}2^{i\eta}\chi_{\Omega_{i,1}\backslash\Omega_{i+1,1}}(x)\right)^{\frac{1}{\eta}}.
\]
Hence, together with Corollary~{\upshape\ref{co2.6}}
\[
\begin{aligned}
    &\left \|\left(\sum\limits_{i}2^{i\eta}\chi_{ \Omega_{i,1}}\right)^{\frac{1}{\eta}}\right \|_{L_{\omega}^{p}}^{p}\leq C\left \|\left(\sum\limits_{i}2^{i\eta}\chi_{\Omega_{i,1}\backslash\Omega_{i+1,1}}\right)^{\frac{1}{\eta}}\right \|_{L_{\omega}^{p}}^{p}\\
&=C\int_{\mathbb{R}^{n}}\left(\sum\limits_{i}2^{i}\chi_{\Omega_{i,1}\backslash\Omega_{i+1,1}}\right)^{p}\omega(x)dx=C\sum\limits_{i}\int_{\Omega_{i,1}\backslash\Omega_{i+1,1}}2^{ip}\omega(x)dx\\
&\leq C\int_{\mathbb{R}^{n}}\left(S^{1}(f)\right)^{p}\omega(x)dx\leq C\left \|f\right \|_{h_{\omega}^{p}}^{p}.
\end{aligned}
\]\par
Next we estimate $\uppercase\expandafter{\romannumeral1}.$ We can use $P$ to denote Q if $Q\in \Pi_{N}$ and rewrite
\[
\uppercase\expandafter{\romannumeral1}=:\sum\limits_{i}\sum\limits_{P\in B_{i,0}}\mu_{P}^{i}b_{P}^{i}(x)
\]
where $b_{P}^{i}(x)=\frac{1}{\mu_{P}^{i}}\vert P\vert(\psi_{0}\ast h)(u_{P})\psi_{0}(x-u_{P})$ and $\mu_{P}^{i}=\widetilde C\vert(\psi_{0}\ast h)(u_{P})\vert$. Let $\widetilde C=2^{-Nn}\omega(P)^{\frac{1}{p}}\vert P\vert^{-\frac{1}{q}}\left \|\psi_{0}\right \|_{L^{q}}$.\par
Similarly, by the definition of $\psi_{0}$, we find that $b_{P}^{i}$ is supported in $c_{0}P$ where $c_{0}=2^{N+2}$ and the vanishing moment condition of $b_{P}^{i}$ follows from the vanishing moment condition of $\psi_{0}$. Moreover, there exist a constant $C\geq1$ such that $\vert c_{0}P\vert>C$. It's easy to prove that $\left \|b_{P}^{i}\right \|_{L^{q}}=\frac{1}{\mu_{P}^{i}}\vert P\vert\vert\psi_{0}\ast h(u_{P})\vert(\int\vert\psi_{0}(x-u_{P})\vert^{q}dx)^{\frac{1}{q}}\leq\vert P\vert^{\frac{1}{q}}\omega(P)^{-\frac{1}{p}}.$\par
In conclusion, each $b_{P}^{i}(x)$ is a $\omega\text{-}(p,q,s)\text{-}$block of $h_\omega^{p}(\mathbb R^{n})$.\par
Note that $P_{i}\subset\widetilde\Omega_{i,0}$ when $P\in B_{i,0}$.
Repeating the similar but easier argument, we can obtain
\[
\left \|\left(\sum\limits_{i}\sum\limits_{P\in B_{i,0}}\left(\frac{\mu_{P}^{i} \chi_{c_{0}P_i}}{\omega(P)^{\frac{1}{p}}}\right)^{\eta}\right)^{\frac{1}{\eta}}\right \|_{L_{\omega}^{p}}\leq C\left \|f\right \|_{h_{\omega}^{p}}.
\]\\
Consequently, we can know that
\[
\left \|\left(\sum\limits_{i}\sum\limits_{\widetilde Q\in B_{i,1}}\left(\frac{\lambda_{\widetilde Q}^{i} \chi_{c_{1}\widetilde Q_i}}{\omega(\widetilde Q)^{\frac{1}{p}}}\right)^{\eta}\right)^{\frac{1}{\eta}}\right \|_{L_{\omega}^{p}}+\left \|\left(\sum\limits_{i}\sum\limits_{P\in B_{i,0}}\left(\frac{\mu_{P}^{i} \chi_{c_{0}P_i}}{\omega(P)^{\frac{1}{p}}}\right)^{\eta}\right)^{\frac{1}{\eta}}\right \|_{L_{\omega}^{p}}\leq C\left \|f\right \|_{h_{\omega}^{p}}.
\]
\end{proof}
Next we will prove the reconstruction theorem for the atomic decomposition.
\begin{proof}[Proof of Theorem~{\upshape\ref{th1.8}}]
Notice that for almost every $x\in\mathbb{R}^{n}$
\[
\vert g(f)(x)\vert\leq\sum\limits_{j=1}^{\infty}\lambda_{j}\vert g(a_{j})(x)\vert+\sum\limits_{j=1}^{\infty}\mu_{j}\vert g(b_{j})(x)\vert=\uppercase\expandafter{\romannumeral1}+\uppercase\expandafter{\romannumeral2}.
\]\par
For $\uppercase\expandafter{\romannumeral2}$,
\[
\begin{aligned}
\uppercase\expandafter{\romannumeral2}
&=\sum\limits_{j=1}^{\infty}\mu_{j}\vert g(b_{j})(x)\vert\chi_{4P_{j}}(x)+\sum\limits_{j=1}^{\infty}\mu_{j}\vert g(b_{j})(x)\vert\chi_{(4P_{j})^{c}}(x)\\
&=\uppercase\expandafter{\romannumeral1}_{1}+\uppercase\expandafter{\romannumeral1}_{2}.
\end{aligned}
\]\par
Now we estimate the term $\uppercase\expandafter{\romannumeral1}_{1}$. Denote $h_{j}(x)=g(b_{j})(x)\chi_{4P_{j}}$. By the size condition of atoms and $q>q_{r}$, we obtain
\[
\left \|h_{j}\right \|_{L^{q}}\leq\left \|g(b_{j})\right \|_{L^{q}}\leq C\left \|b_{j}\right \|_{L^{q}}\leq C\vert P_{j}\vert^{\frac{1}{q}}\omega(P_{j})^{-\frac{1}{p}}.
\]
Together with the fact ${\rm supp}(h_{j})\subset 4P_{j}$ and Lemma~{\upshape\ref{le2.3}}, we obtain
\[
\left \|\uppercase\expandafter{\romannumeral1}_{1}\right \|_{L_{\omega}^{p}}=\left \|\sum\limits_{j=1}^{\infty}\mu_{j}h_{j}(x)\right \|_{L_{\omega}^{p}}\leq C\left \|\sum\limits_{j=1}^{\infty}\frac{\mu_{j}\chi_{4P_{j}}}{\omega(P_{j})^{\frac{1}{p}}}\right \|_{L_{\omega}^{p}}.
\]\\
Assume that $\gamma$ is a large constant satisfying $\gamma p>q_{\omega}$. And it's easy to prove that $\chi_{4P_{j}}(x)\leq CM^{\gamma}(\chi_{P_{j}})(x)$. Then by Lemma~{\upshape\ref{le2.1}}, we obtain
\[
\left \|\sum\limits_{j=1}^{\infty}\frac{\mu_{j}\chi_{4P_{j}}}{\omega(P_{j})^{\frac{1}{p}}}\right \|_{L_{\omega}^{p}}\leq C\left \|\left(\sum\limits_{j=1}^{\infty}\frac{\mu_{j}M^{\gamma}(\chi_{P_{j}})}{\omega(P_{j})^{\frac{1}{p}}}\right)^{\frac{1}{\gamma}}\right \|_{L_{\omega}^{\gamma p}}^{\gamma}\leq C\left \|\sum\limits_{j=1}^{\infty}\frac{\mu_{j}\chi_{P_{j}}}{\omega(P_{j})^{\frac{1}{p}}}\right \|_{L_{\omega}^{p}}
\]
which implies that
\[
\left \|\uppercase\expandafter{\romannumeral1}_{1}\right \|_{L_{\omega}^{p}}\leq C\left \|\sum\limits_{j=1}^{\infty}\frac{\mu_{j}\chi_{P_{j}}}{\omega(P_{j})^{\frac{1}{p}}}\right \|_{L_{\omega}^{p}}.
\]\par
Next we estimate $\uppercase\expandafter{\romannumeral1}_{2}$. For all $x\in(4P_{j})^{c}$, we have
\[
\begin{aligned}
\vert(\phi_{i}\ast b_{j})(x)\vert
&\leq\int_{P_{j}}\vert\phi_{i}(x-y)b_{j}(y)\vert dy\\
&\leq \sup\limits_{z\in P_{j}}\vert\phi_{i}(x-z)\vert\int_{P_{j}}\vert b_{j}(y)\vert dy\\
&\leq C\frac{2^{in}}{(1+2^{i}\vert x-x_{P_{j}}\vert)^{M}}\left \|b_{j}\right \|_{L^{q}}\vert P_{j}\vert^{\frac{1}{q^{\prime}}}\\
&\leq C\frac{2^{in}}{(1+2^{i}\vert x-x_{P_{j}}\vert)^{M}}\frac{\vert P_{j}\vert^{\frac{1}{q}}\vert P_{j}\vert^{\frac{1}{q^{\prime}}}}{\omega(P_{j})^{\frac{1}{p}}}\\
&\leq C\frac{2^{in}}{(1+2^{i}\vert x-x_{P_{j}}\vert)^{M}}\frac{l(P_{j})^{M}}{\omega(P_{j})^{\frac{1}{p}}}
\end{aligned}
\]
for some sufficient large $M>n>0$. Observe that $\vert P_{j}\vert> C\geq1$ and if $M>n$,
\[
\sum\limits_{j=0}^{\infty}C\frac{2^{in}}{(1+2^{i}\vert x-x_{P_{j}}\vert)^{M}}\leq\frac{C}{\vert x-x_{P_{j}}\vert^{M}}.
\]
Therefore, we obtain
\[
\begin{aligned}
\uppercase\expandafter{\romannumeral1}_{2}
&=\sum\limits_{j=1}^{\infty}\mu_{j}\left\{\sum_{i\in\mathbb{N}}\vert\phi_{i}\ast b_{j}(x)\vert^{2}\right\}^{\frac{1}{2}}\chi_{(4P_{j})^{c}}(x)\\
&\leq\sum\limits_{j=1}^{\infty}\mu_{j}\left(\sum_{i\in\mathbb{N}}\vert\phi_{i}\ast b_{j}(x)\vert\right)\chi_{(4P_{j})^{c}}(x)\\
&\leq C\sum\limits_{j=1}^{\infty}\mu_{j}\frac{{\omega(P_{j})}^{-\frac{1}{p}}(l(P_{j}))^{M}}{\vert x-x_{P_{j}}\vert^{M}}\chi_{(4P_{j})^{c}}(x).\\
\end{aligned}
\]
Let $M=n+s+1$ and $\gamma=\frac{M}{n}$. We have
\[
\begin{aligned}
\uppercase\expandafter{\romannumeral1}_{2}
&\leq C\sum\limits_{j=1}^{\infty}\mu_{j}{\omega(P_{j})}^{-\frac{1}{p}}{\left({\left(\frac{l(P_{j})}{\vert x-x_{P_{j}}\vert}\right)}^{n}\right)}^{\gamma}\chi_{(4P_{j})^{c}}(x)\\
&\leq C\sum\limits_{j=1}^{\infty}\mu_{j}{\omega(P_{j})}^{-\frac{1}{p}}(M\chi_{P_{j}})^{\gamma}(x)
\end{aligned}
\]
Since $\omega\in A_{\infty}$ with the critical index $q_{\omega}$ and $s\geq {\rm max}\{[n(\frac{q_{\omega}}{p}-1)],-1\}$, we know that $\gamma p>q_{\omega}$ and then $\omega\in A_{\gamma p}$. Applying Lemma~{\upshape\ref{le2.1}} yields that
\[
\begin{aligned}
    \left \|\uppercase\expandafter{\romannumeral1}_{2}\right \|_{L_{\omega}^{p}}&\leq C\left \|\sum\limits_{j=1}^{\infty}\frac{\mu_{j}(M\chi_{P_{j}})^{\gamma}}{\omega(P_{j})^{\frac{1}{p}}}\right \|_{L_{\omega}^{p}}=C\left \|\left(\sum\limits_{j=1}^{\infty}\frac{\mu_{j}(M\chi_{P_{j}})^{\gamma}}{\omega(P_{j})^{\frac{1}{p}}}\right)^{\frac{1}{\gamma}}\right \|_{L_{\omega}^{\gamma p}}^{\gamma}\\
&\leq C\left \|\left(\sum\limits_{j=1}^{\infty}\frac{\mu_{j}\chi_{P_{j}}}{\omega(P_{j})^{\frac{1}{p}}}\right)^{\frac{1}{\gamma}}\right \|_{L_{\omega}^{\gamma p}}^{\gamma}= C\left \|\sum\limits_{j=1}^{\infty}\frac{\mu_{j}\chi_{P_{j}}}{\omega(P_{j})^{\frac{1}{p}}}\right \|_{L_{\omega}^{p}}.
\end{aligned}
\]
By Lemma~{\upshape\ref{le2.7}} and the estimates of $\uppercase\expandafter{\romannumeral1}_{1}$ and $\uppercase\expandafter{\romannumeral1}_{2}$, we obtain
\[
\left \|\sum\limits_{j=1}^{\infty}\mu_{j}b_{j}\right \|_{h_{\omega}^{p}}\leq C\left \|\sum\limits_{j=1}^{\infty}\mu_{j}g(b_{j})\right \|_{L_{\omega}^{p}}
\leq C\left \|\sum\limits_{j=1}^{\infty}\frac{\mu_{j}\chi_{P_{j}}}{\omega(P_{j})^{\frac{1}{p}}}\right \|_{L_{\omega}^{p}}.
\]\par
Similarly, for $\uppercase\expandafter{\romannumeral1}$, we can find that
\[
\begin{aligned}
\uppercase\expandafter{\romannumeral1}
&=\sum\limits_{j=1}^{\infty}\lambda_{j}\vert g(a_{j})(x)\vert\chi_{2Q_{j}}(x)+\sum\limits_{j=1}^{\infty}\lambda_{j}\vert g(a_{j})(x)\vert\chi_{(2Q_{j})^{c}}(x)\\
&:=\uppercase\expandafter{\romannumeral1}_{1}+\uppercase\expandafter{\romannumeral1}_{2}
\end{aligned}
\]
and
\[
\left \|\uppercase\expandafter{\romannumeral1}_{1}\right \|_{L_{\omega}^{p}}\leq C\left \|\sum\limits_{j=1}^{\infty}\frac{\lambda_j \chi_{Q_j}}{\omega(Q_j)^{\frac{1}{p}}}\right \|_{L_{\omega}^{p}}.
\]
Note that $P_{j}\equiv P_{j}(x_{j},l(P_{j}))$. Denote by $p^{s}_{i}$ the sum of first $s+1$ terms in the Taylor expansion of $\phi_{i}(y-z)$ at $y-x_{j}$. Details are similar to those in \cite{tan2022revisit}. Applying the vanishing moment and size condition of $a_{j}$ and the smoothness conditions on $\phi_{j}$, we can obtain
\[
\begin{aligned}
    g(a_{j})^{2}(x)
    &=\sum\limits_{i}\left\vert\int_{\mathbb {R}^{n}}a_{j}(z)[\phi_{i}(y-z)-p^{s}_{i}(y,z,x_{j})]dz\right\vert^{2}\\
&\leq C\frac{\omega(Q_{j})^{-\frac{2}{p}}l(Q_{j})^{2(n+s+1)}}{\vert y-x_{j}\vert^{2(n+s+1)}}.
\end{aligned}
\]
Let $\gamma=\frac{n+s+1}{n}$. By repeating the similar analysis as in the estimate of $\uppercase\expandafter{\romannumeral2}$, we can obtain
\[
\left \|\uppercase\expandafter{\romannumeral1}_{2}\right \|_{L_{\omega}^{p}}\leq C\left \|\sum\limits_{j=1}^{\infty}\frac{\lambda_j \chi_{Q_j}}{\omega(Q_j)^{\frac{1}{p}}}\right \|_{L_{\omega}^{p}}.
\]
Therefore, it concludes that
\[
\left \|\sum\limits_{j=1}^{\infty}\lambda_{j}a_{j}\right \|_{h_{\omega}^{p}}\leq C\left \|\sum\limits_{j=1}^{\infty}\frac{\lambda_{j}\chi_{Q_{j}}}{\omega(Q_{j})^{\frac{1}{p}}}\right \|_{L_{\omega}^{p}}
\]\par
Observe that
\[
\left \|\sum\limits_{j=1}^{\infty}\frac{\lambda_j \chi_{Q_j}}{\omega(Q_j)^{\frac{1}{p}}}\right \|_{L_{\omega}^{p}}+\left \|\sum\limits_{j=1}^{\infty}\frac{\mu_j \chi_{P_j}}{\omega(P_j)^{\frac{1}{p}}}\right \|_{L_{\omega}^{p}}<\infty,
\]\\
which implies that
\[
\left \|\sum\limits_{j=N}^{\infty}\frac{\lambda_j \chi_{Q_j}}{\omega(Q_j)^{\frac{1}{p}}}\right \|_{L_{\omega}^{p}}+\left \|\sum\limits_{j=N}^{\infty}\frac{\mu_j \chi_{P_j}}{\omega(P_j)^{\frac{1}{p}}}\right \|_{L_{\omega}^{p}}\to0,\ as\ N\to\infty.
\]
Thus,
\[
\lim_{N\to\infty}\left \|\sum\limits_{j=N}^{\infty}\frac{\lambda_j \chi_{Q_j}}{\omega(Q_j)^{\frac{1}{p}}}\right \|_{L_{\omega}^{p}}=0,\ \lim_{N\to\infty}\left \|\sum\limits_{j=N}^{\infty}\frac{\mu_j \chi_{P_j}}{\omega(P_j)^{\frac{1}{p}}}\right \|_{L_{\omega}^{p}}=0.
\]
Notice that
\[
\left \|\sum\limits_{j=N}^{\infty}\lambda_{j}a_{j}\right \|_{h_{\omega}^{p}}\leq C\left \|\sum\limits_{j=N}^{\infty}\frac{\lambda_j \chi_{Q_j}}{\omega(Q_j)^{\frac{1}{p}}}\right \|_{L_{\omega}^{p}};
\]
\[
\left \|\sum\limits_{j=N}^{\infty}\mu_{j}b_{j}\right \|_{h_{\omega}^{p}}\leq C\left \|\sum\limits_{j=N}^{\infty}\frac{\mu_j \chi_{P_j}}{\omega(P_j)^{\frac{1}{p}}}\right \|_{L_{\omega}^{p}}.
\]\\
Therefore, we can obtain
\[
\lim_{N\to\infty}\left \|\sum\limits_{j=N}^{\infty}\lambda_{j}a_{j}\right \|_{h_{\omega}^{p}}=0,\ \lim_{N\to\infty}\left \|\sum\limits_{j=N}^{\infty}\mu_{j}b_{j}\right \|_{h_{\omega}^{p}}=0
\]
which implies that the series $\sum\limits_{j}\lambda_{j}a_{j}+\sum\limits_{j}\mu_{j}b_{j}$ converges in $h_{\omega}^{p}(\mathbb{R}^{n}).$
\end{proof}

\section{Proofs of Theorems 1.9, 1.10 and 1.12}\label{sec4}
This section is devoted to proving the boundedness results given in Theorem~{\upshape\ref{th1.9}} and {\upshape\ref{th1.10}} for the inhomogenous Calder\'on-Zygmund singular integrals and Theorem~{\upshape\ref{th1.12}} for the local fractional integrals.
\begin{proof}[Proof of Theorem~{\upshape\ref{th1.9}}]
Recalling the atomic decomposition of weighted local Hardy spaces in Theorem~{\upshape\ref{th1.7}}, we know that if $f\in h_\omega^{p}(\mathbb R^{n})\cap L^{q}(\mathbb R^{n})$, there exist a sequence of $\omega\text{-}(p,q,s)\text{-}$atoms $\{a_{j}\}_{j=1}^{\infty}$ with a corresponding sequence of non-negative numbers $\{\lambda_{j}\}_{j=1}^{\infty}$ and a sequence of $\omega\text{-}(p,q,s)\text{-}$blocks $\{b_{j}\}_{j=1}^{\infty}$ with a corresponding sequence of non-negative numbers $\{\mu_{j}\}_{j=1}^{\infty}$ such that $f=\sum\limits_{j}\lambda_{j}a_{j}+\sum\limits_{j}\mu_{j}b_{j}$ in $h_\omega^{p}(\mathbb R^{n})\cap L^{q}(\mathbb R^{n})$ with $(\frac{n+\eta}{n})p<q<\infty$, and that
\[
\left \|\sum\limits_{j=1}^{\infty}\frac{\lambda_j \chi_{Q_j}}{\omega(Q_j)^{\frac{1}{p}}}\right \|_{L_{\omega}^{p}}+\left \|\sum\limits_{j=1}^{\infty}\frac{\mu_j \chi_{P_j}}{\omega(P_j)^{\frac{1}{p}}}\right \|_{L_{\omega}^{p}}\leq C\left \|f\right \|_{h_{\omega}^{p}}
\]
To prove the theorem, it will suffice to prove that
\[
\left \| T(f) \right \|_{L_{\omega}^{p}}\leq C_{1}\left \|\sum\limits_{j=1}^{\infty}\frac{\lambda_j \chi_{Q_j}}{\omega(Q_j)^{\frac{1}{p}}}\right \|_{L_{\omega}^{p}}+C_{2}\left \|\sum\limits_{j=1}^{\infty}\frac{\mu_j \chi_{P_j}}{\omega(P_j)^{\frac{1}{p}}}\right \|_{L_{\omega}^{p}}.
\]\par
In fact, for $x\in \mathbb R^{n}$, we have
\[
\left\vert T(f)(x) \right\vert\leq\sum\limits_{j}\vert\lambda_{j}\vert\vert T(a_{j})(x)\vert+\sum\limits_{j}\vert\mu_{j}\vert\vert T(b_{j})(x)\vert=:\uppercase\expandafter{\romannumeral1}+\uppercase\expandafter{\romannumeral2}.
\]\par
First we can prove that
\[
\left \| \uppercase\expandafter{\romannumeral1} \right \|_{L_{\omega}^{p}}\leq C\left \|\sum\limits_{j=1}^{\infty}\frac{\lambda_j \chi_{Q_j}}{\omega(Q_j)^{\frac{1}{p}}}\right \|_{L_{\omega}^{p}}.
\]
For $x\in\mathbb R^{n}$,
\[
\uppercase\expandafter{\romannumeral1}\leq\sum\limits_{j=1}^{\infty}\vert \lambda_{j}\vert\vert T(a_{j})(x)\vert\chi_{Q_{j}^{*}}(x)+\sum\limits_{j=1}^{\infty}\vert\lambda_{j}\vert\vert T(a_{j})(x)\vert\chi_{(Q_{j}^{*})^{c}}(x)=:\uppercase\expandafter{\romannumeral1}_{1}+\uppercase\expandafter{\romannumeral1}_{2}.
\]
Since $T$ is a bounded operator on $L^{2}$, from the Calder\'on-Zygmund real method in \cite[Section 7.3]{Meyer1990}, we know that $T$ is bounded on $L^{q}$ for any $1<q<\infty$. Together with the size condition of $a_{j}$, we obtain that for any $(\frac{n+\eta}{n})p<q<\infty$
\[
\left ( \frac{1}{\vert Q_{j}\vert} \int_{Q_{j}}\vert T(a_{j})(x)\vert^{q}dx\right)^{\frac{1}{q}}\leq\frac{\left \| a_{j} \right \|_{L^{q}}}{\vert Q_{j}\vert^{1/q}}\leq\frac{1}{\omega(Q_{j})^{1/p}}.
\]
Since $\omega\in A_{(\frac{n+\eta}{n})p}$, then there exists $r>1$ such that $\omega\in RH_{r}$. Fix $q_{0}>\max\{\frac{n+\eta}{n}p,\frac{r}{r-1}p\}$ such that $(\frac{q_{0}}{p})^{\prime}<r$. For $\uppercase\expandafter{\romannumeral1}_{1}$, by Lemma~{\upshape\ref{le2.8}} and Remark~{\upshape\ref{re2.2}}, we can get that
\[
\begin{aligned}
        \left \| \uppercase\expandafter{\romannumeral1}_{1} \right \|_{L_{\omega}^{p}}
        &\leq\left \| \sum\limits_{j}\vert\lambda_{j}\vert\vert T(a_{j})\vert\chi_{Q_{j}^{*}} \right \|_{L_{\omega}^{p}}\\
        &C\leq\left \| \sum\limits_{j}\vert\lambda_{j}\vert\left ( \frac{1}{\vert Q_{j}\vert} \int_{Q_{j}}\vert T(a_{j})(x)\vert^{q_{0}}dx\right)^{\frac{1}{q_{0}}}\chi_{Q_{j}^{*}} \right \|_{L_{\omega}^{p}}\\
        &\leq C\left \| \sum\limits_{j}\lambda_{j}\frac{\chi_{Q_{j}^{*}}}{\omega(Q_{j})^{1/p}} \right \|_{L_{\omega}^{p}}\\
        &\leq C\left \| \sum\limits_{j}\frac{\lambda_{j}\chi_{Q_{j}}}{\omega(Q_{j})^{1/p}} \right \|_{L_{\omega}^{p}}.
\end{aligned}
\]
For $\uppercase\expandafter{\romannumeral1}_{2}$, note that $x\in(Q_{j}^{*})^{c}$ and $c_{Q_{j}}$ is the center of $Q_{j}$. We can know that $\vert x-c_{Q_{j}}\vert\geq2\vert y-c_{Q_{j}}\vert$ and $\vert y-c_{Q_{j}}\vert\leq l(Q_{j})$. Applying the smooth of condition the kernel $\mathcal{K}$, we obtain that
\[
\begin{aligned}
\vert T(a_{j})(x)\vert
&=\left\vert\int_{Q_{j}}\mathcal{K}(x,y)a_{j}(y)dy\right\vert\\
&\leq\int_{Q_{j}}\vert \mathcal{K}(x,y)-\mathcal{K}(x,c_{Q_{j}})\vert\vert a_{j}(y)\vert dy\\
&\leq C\int_{Q_{j}}\frac{\vert y-c_{Q_{j}}\vert^{\epsilon}}{\vert x-c_{Q_{j}}\vert^{n+\epsilon}}\vert a_{j}(y)\vert dy\\
&\leq C\frac{l(Q_{j})^{\epsilon}}{\vert x-c_{Q_{j}}\vert^{n+\epsilon}} \left \| a_{j} \right \|_{L^{q}}\vert Q_{j}\vert^{\frac{1}{q^{\prime}}}\\
&\leq C\frac{l(Q_{j})^{n+\epsilon}}{\omega(Q_{j})^{\frac{1}{p}}\vert x-c_{Q_{j}}\vert^{n+\epsilon}}\\
&\leq C\frac{(M(\chi_{Q_{j}})(x))^{\frac{n+\eta}{n}}}{\omega(Q_{j})^{\frac{1}{p}}}.
\end{aligned}
\]
Denote that $\gamma=\frac{n+\eta}{n}$. Applying Fefferman-Stein vector-valued maximal inequality yields that
\[
\begin{aligned}
\left \| \uppercase\expandafter{\romannumeral1}_{2} \right \|_{L_{\omega}^{p}}
&\leq C\left \| \sum\limits_{j}\frac{\vert\lambda_{j}\vert M^{\gamma}(\chi_{Q_{j}})}{\omega(Q_{j})^{\frac{1}{p}}} \right \|_{L_{\omega}^{p}}\\
&\leq C\left \| \left(\sum\limits_{j}\frac{\lambda_{j} M^{\gamma}(\chi_{Q_{j}})}{\omega(Q_{j})^{\frac{1}{p}}}\right)^{\frac{1}{\gamma}} \right \|_{L_{\omega}^{\gamma p}}^{\gamma}\\
&\leq C\left \| \sum\limits_{j}\frac{\lambda_{j}\chi_{Q_{j}}}{\omega(Q_{j})^{1/p}} \right \|_{L_{\omega}^{p}}.
\end{aligned}
\]
Combining the estimates of $\uppercase\expandafter{\romannumeral1}_{1}$ and $\uppercase\expandafter{\romannumeral1}_{2}$, we can obtain the desired result.\par
Then we can prove that
\[
\left \| \uppercase\expandafter{\romannumeral2} \right \|_{L_{\omega}^{p}}\leq C\left \|\sum\limits_{j=1}^{\infty}\frac{\mu_j \chi_{P_j}}{\omega(P_j)^{\frac{1}{p}}}\right \|_{L_{\omega}^{p}}.
\]
By repeating the similar argument, we can know that
for $x\in\mathbb R^{n}$,
\[
\uppercase\expandafter{\romannumeral2}\leq\sum\limits_{j=1}^{\infty}\vert \mu_{j}\vert\vert T(b_{j})(x)\vert\chi_{P_{j}^{*}}(x)+\sum\limits_{j=1}^{\infty}\vert\mu_{j}\vert\vert T(b_{j})(x)\vert\chi_{(P_{j}^{*})^{c}}(x)=:\uppercase\expandafter{\romannumeral1}_{1}+\uppercase\expandafter{\romannumeral1}_{2}
\]
and
\[
\left \| \uppercase\expandafter{\romannumeral1}_{1} \right \|_{L_{\omega}^{p}}\leq C\left \|\sum\limits_{j=1}^{\infty}\frac{\mu_j \chi_{P_j}}{\omega(P_j)^{\frac{1}{p}}}\right \|_{L_{\omega}^{p}}.
\]
For $\uppercase\expandafter{\romannumeral1}_{2}$, when $x\in(P_{j}^{*})^{c}$ and $y\in P_{j}$, we have $\vert x-y\vert\sim\vert x-c_{P_{j}}\vert$ and $\vert x-y\vert>1/2$. By using the size condition of $\mathcal{K}$ and the fact that $\vert P_{j}\vert>C$, we can get that for any $x\in(P_{j}^{*})^{c}$, 
\[
\begin{aligned}
\vert T(b_{j})(x)\vert
&=\left\vert\int_{P_{j}}\mathcal{K}(x,y)b_{j}(y)dy\right\vert\\
&=\int_{Q_{j}}\vert\mathcal{K}(x,y)\vert\vert b_{j}(y)\vert dy\\
&\leq C\frac{1}{\vert x-c_{P_{j}}\vert^{n+\delta}} \left \| b_{j} \right \|_{L^{q}}\vert P_{j}\vert^{\frac{1}{q^{\prime}}}\\
&\leq C\frac{l(P_{j})^{n}}{\omega(P_{j})^{1/p}\vert x-c_{P_{j}}\vert^{n+\delta}}\\
&\leq C\frac{l(P_{j})^{n+\delta}}{\omega(P_{j})^{1/p}\vert x-c_{P_{j}}\vert^{n+\delta}}\\
&\leq C\frac{(M(\chi_{P_{j}})(x))^{\frac{n+\eta}{n}}}{\omega(P_{j})^{\frac{1}{p}}}.
\end{aligned}
\]
Then, it concludes that
\[
\left \| \uppercase\expandafter{\romannumeral1}_{2} \right \|_{L_{\omega}^{p}}\leq C\left \| \sum\limits_{j}\frac{\lambda_{j}\chi_{P_{j}}}{\omega(P_{j})^{1/p}} \right \|_{L_{\omega}^{p}}.
\]\par
Therefore, by a density argument,we finish the proof of the theorem.
\end{proof}

\begin{proof}[Proof of Theorem~{\upshape\ref{th1.10}}]
By the argument similar to that used in the above proof, it will suffice to prove that for $f\in h_\omega^{p}(\mathbb R^{n})\cap L^{q}(\mathbb R^{n})$ and $(\frac{n+\eta}{n})p<q<\infty$,
\[
\left \| T(f) \right \|_{h_{\omega}^{p}}=\left \| M_{\Phi}(T(f)) \right \|_{L_{\omega}^{p}}\leq C_{1}\left \|\sum\limits_{j=1}^{\infty}\frac{\lambda_j \chi_{Q_j}}{\omega(Q_j)^{\frac{1}{p}}}\right \|_{L_{\omega}^{p}}+C_{2}\left \|\sum\limits_{j=1}^{\infty}\frac{\mu_j \chi_{P_j}}{\omega(P_j)^{\frac{1}{p}}}\right \|_{L_{\omega}^{p}}.
\]
We claim that for $x\in\mathbb R^{n}$, we have
\[
\sup\limits_{0<t<1}\vert\Phi_{t}\ast T(f)(x)\vert\leq\uppercase\expandafter{\romannumeral1}+\uppercase\expandafter{\romannumeral2}
\]
where 
\[
\begin{aligned}
\uppercase\expandafter{\romannumeral1}=\sum\limits_{j}\frac{\lambda_{j}}{\omega(Q_{j})^{1/p}}(M(T(a_{j}))(x)\chi_{2\sqrt{n}Q_{j}^{*}}(x)
+(M(\chi_{Q_{j}})(x))^{\gamma}\chi_{(2\sqrt{n}Q_{j}^{*})^{c}}(x))
\end{aligned}
\]
and
\[
\begin{aligned}
\uppercase\expandafter{\romannumeral2}=\sum\limits_{j}\frac{\mu_{j}}{\omega(P_{j})^{1/p}}(M(T(b_{j}))(x)\chi_{2\sqrt{n}P_{j}^{*}}(x)
+(M(\chi_{P_{j}})(x))^{\gamma}\chi_{(2\sqrt{n}P_{j}^{*})^{c}}(x))
\end{aligned}
\]
with $\gamma=\frac{n+\eta}{n}$. Applying the claim and repeating the nearly identical argument to the proof of Theorem~{\upshape\ref{th1.9}}, we can obtain the desired result. In fact, when $x\in 2\sqrt{n}Q_{j}^{*}$, we just need the pointwise estimate
\[
M_{\Phi}(T(\sum\limits_{j}\lambda_{j}a_{j}))(x)\leq C\sum\limits_{j}\lambda_{j}M(T(a_{j}))(x).
\]
When $x\in(2\sqrt{n}Q_{j}^{*})^{c}$, we have
\[
\begin{aligned}
    \vert\Phi_{j}\ast T(a_{j})(x)\vert
    &=\left\vert \int_{\mathbb R^{n}}\Phi_{t}(x-y)T(a_{j})(y)dy\right\vert\\
    &\leq t^{-n}\int_{B(x,t)}\vert T(a_{j})(y)\vert dy\\
    &\leq\sup\limits_{y\in B(x,t)}\vert T(a_{j})(y)\vert.
\end{aligned}
\]
Notice that $\vert Q_{j}\vert\leq C$ and $x\in(2\sqrt{n}Q_{j}^{*})^{c}$. If $0<t\leq\vert x-c_{Q_{j}}\vert/2$, we can get that $y\in(Q_{j}^{*})^{c}$. Therefore, from the proof of Theorem~{\upshape\ref{th1.9}}, we conclude that
\[
\sup\limits_{y\in B(x,t)}\vert T(a_{j})(y)\vert\leq C\frac{(M(\chi_{Q_{j}})(x))^{\gamma}}{\omega(Q_{j})^{\frac{1}{p}}}.
\]
Then we consider the case that $t>\vert x-c_{Q_{j}}\vert/2$. Observe that $a_{j}$ satisfies $\int_{\mathbb R^{n}}T(a_{j})(x)dx=0$. For any $x\in(2\sqrt{n}Q_{j}^{*})^{c}$, applying the mean value theorem and H\"older's inequality yields that
\[
\begin{aligned}
&\vert\Phi_{t}\ast T(a_{j})(x)\vert
=\left\vert\int_{\mathbb R^{n}}(\Phi_{t}(x-y)-\Phi_{t}(x-c_{Q_{j}}))T(a_{j})(x)dy\right\vert\\
&\leq t^{-n}\int_{\mathbb R^{n}}\left\vert\frac{y-c_{Q_{j}}}{t}\right\vert\vert\Phi^{\prime}((x-c_{Q_{j}}+\theta(c_{Q_{j}}-y))/t)\vert\vert T(a_{j})(y)\vert dy\\
&\leq C\vert x-c_{Q_{j}}\vert^{-n-1}\left(\int_{Q_{j}^{*}}\vert y-c_{Q_{j}}\vert\vert T(a_{j})(y)\vert dy+\int_{(Q_{j}^{*})^{c}}\vert y-c_{Q_{j}}\vert\vert T(a_{j})(y)\vert dy\right)\\
&\leq C\vert x-c_{Q_{j}}\vert^{-n-1}\left(l(Q_{j})^{\frac{n}{q^{\prime}}+1}\left \| T(a_{j}) \right \|_{L^{q}}+\int_{(Q_{j}^{*})^{c}}\frac{l(Q_{j})^{n+\eta}}{\omega(Q_{j})^{\frac{1}{p}}\vert y-c_{Q_{j}}\vert^{n+\eta-1}}dy\right)\\
&\leq C\vert x-c_{Q_{j}}\vert^{-n-1}l(Q_{j})^{n+1}\omega(Q_{j})^{-\frac{1}{p}}\leq C\frac{(M(\chi_{Q_{j}})(x))^{\gamma}}{\omega(Q_{j})^{\frac{1}{p}}}
\end{aligned}
\]
where $\theta\in(0,1)$.\par
Similarly, when $x\in2\sqrt{n}P_{j}^{*}$, we can get that
\[
M_{\Phi}(T(\sum\limits_{j}\mu_{j}b_{j})(x))\leq C\sum\limits_{j}\mu_{j}M(T(b_{j}))(x).
\]
and when $x\in(2\sqrt{n}P_{j}^{*})^{c}$, 
\[
\vert\Phi_{j}\ast T(b_{j})(x)\vert\leq\sup\limits_{y\in B(x,t)}\vert T(b_{j})(y)\vert.
\]
Notice that $\vert P_{j}\vert>C$. Then we can get that $y\in(P_{j}^{*})^{c}$. Then, repeating the same argument as used above, we can obtain that
\[
\sup\limits_{y\in B(x,t)}\vert T(b_{j})(y)\vert\leq C\frac{(M(\chi_{P_{j}})(x))^{\gamma}}{\omega(P_{j})^{\frac{1}{p}}}.
\]\par
Therefore, we complete the proof of the claim so that we can obtain the theorem.
\end{proof}

\begin{proof}[Proof of Theorem~{\upshape\ref{th1.12}}]
To prove the first part of this theorem, we apply the argument similar to that used in the proof of Theorem~{\upshape\ref{th1.9}} and Theorem~{\upshape\ref{th1.10}} and so we only need to concentrate on the differences. Now we consider the case when $1<q<\infty$. By the atomic decomposition of $h^{p}_{\omega}$ and a dense argument, in order to show that $I^{loc}_{\alpha}$ admits a bounded extension from $h^{p}_{\omega^{p}}$ to $L^{q}_{\omega^{q}}$, we only need to prove that
\begin{equation}\label{eq1}
\left \| \sum\limits_{j}\lambda_{j}I^{loc}_{\alpha}(a_{j}) \right \|_{L^{q}(\omega^{q})}\leq C_{1}\left \| \sum\limits_{j}\frac{\lambda_{j}\chi_{Q_{j}}}{\omega(Q_{j})^{\frac{1}{p}}} \right \|_{L^{p}(\omega^{p})}
\end{equation}
and
\begin{equation}\label{eq2}
\left \| \sum\limits_{j}\mu_{j}I^{loc}_{\alpha}(b_{j}) \right \|_{L^{q}(\omega^{q})}\leq C_{2}\left \| \sum\limits_{j}\frac{\mu_{j}\chi_{P_{j}}}{\omega(P_{j})^{\frac{1}{p}}} \right \|_{L^{p}(\omega^{p})}
\end{equation}
where each $a_{j}$ is $\omega$\text{-}$(p,t,s)$\text{-}atom, each $b_{j}$ is $\omega$\text{-}$(p,t,s)$\text{-}block and the exact value of $t$ will be chosen below.\par
Then we prove {\upshape\ref{eq1}}. In fact, when $\vert x-c_{Q_{j}}\vert\leq l(Q_{j}^{*})$, by applying the size condition of $a_{j}$, we obtain that
\[
\vert I^{loc}_{\alpha}(a_{j})(x)\vert\leq C\int_{Q_{j}}\frac{1}{\vert x-y\vert^{n-\alpha}}\vert a_{j}(y)\vert dy\leq C \frac{l(Q_{j})^{\alpha}}{\omega(Q_{j})^{\frac{1}{p}}}.
\]
Let $P_{N}(y)$ be the Taylor polynomial of degree $d$ of the kernel of $I^{loc}_{\alpha}$ centered at $c_{Q_{j}}$ where the exact value of $d$ will be chosen below. When $\vert x-c_{Q_{j}}\vert>l(Q_{j}^{*})$, by the moment condition of $a_{j}$ and the Taylor expansion theorem, we obtain that
\[
\begin{aligned}
    \vert I^{loc}_{\alpha}(a_{j})(x)\vert
    &\leq C\int_{Q_{j}}\left\vert\frac{1}{\vert x-y\vert^{n-\alpha}}-P_{N}(y)\right\vert\vert a_{j}(y)\vert dy\\
    &\leq C\int_{Q_{j}}\frac{1}{\vert x-c_{Q_{j}}\vert^{n+d+1-\alpha}}\vert y-c_{Q_{j}}\vert^{d+1}\vert a_{j}(y)\vert dy\\
    &\leq C\frac{l(Q_{j})^{n+d+1}}{\vert x-c_{Q_{j}}\vert^{n+d+1-\alpha}\omega(Q_{j})^{1/p}}\\
    &\leq C\frac{l(Q_{j})^{\alpha}(M\chi_{Q_{j}}(x))^{\gamma}}{\omega(Q_{j})^{1/p}},
\end{aligned}
\]
where $\gamma=\frac{n+d+1-\alpha}{n+1}$. Thus, we can conclude that for $x\in\mathbb R^{n}$,
\[
\vert I^{loc}_{\alpha}(a_{j})(x)\vert\leq C\frac{l(Q_{j})^{\alpha} (M\chi_{Q_{j}}(x))^{\gamma}}{\omega(Q_{j})^{\frac{1}{p}}}.
\]
Since $\omega^{p}\in RH_{\frac{q}{p}}$, then $\omega^{q}\in A_{\infty}$. We choose $d$ such that $\gamma q>q_{\omega^{q}}$.
Therefore, by Fefferman-Stein vector-valued maximal inequality and Lemma~{\upshape\ref{le2.9}}, we have
\[
\begin{aligned}
    \left \| \sum\limits_{j}\lambda_{j}I^{loc}_{\alpha}(a_{j}) \right \|_{L^{q}(\omega^{q})}
    &\leq C\left \| \sum\limits_{j}\lambda_{j}\frac{l(Q_{j})^{\alpha}(M\chi_{Q_{j}}(x))^{\gamma}}{\omega(Q_{j})^{1/p}} \right \|_{L^{q}(\omega^{q})}\\
    &\leq C\left \| \sum\limits_{j}\lambda_{j}\frac{l(Q_{j})^{\alpha}\chi_{Q_{j}}}{\omega(Q_{j})^{1/p}} \right \|_{L^{q}(\omega^{q})}\\
    &\leq C\left \| \sum\limits_{j}\frac{\lambda_{j}\chi_{Q_{j}}}{\omega(Q_{j})^{1/p}} \right \|_{L^{p}(\omega^{p})}.
\end{aligned}
\]\par
Now we prove {\upshape\ref{eq2}}. Notice that
\[
{\rm supp}(I^{loc}_{\alpha}(b_{j}))\subset P_{j}(c_{P_{j}},l(P_{j})+4)\subset 10P_{j}.
\]
Since $\omega^{p}\in RH_{\frac{q}{p}}$, $\omega^{q}\in A_{\infty}$, there exists $r>1$ such that $\omega^{q}\in RH_{r}$. Fix $q_{0}>(\frac{r}{r-1})q$ such that $(\frac{q_{0}}{q})^{\prime}<r$. Then $\omega^{q}\in RH_{\left(\frac{q_{0}}{q}\right)^{\prime}}$. Let $0<\alpha_{0}<\alpha<n$ satisfying $\frac{n}{\alpha_{0}}>q_{\omega^{p}}$. Define $p_{0}>q_{\omega^{p}}$ by $\frac{1}{p_{0}}-\frac{1}{q_{0}}=\frac{\alpha_{0}}{n}$. Moreover, we choose $t=p_{0}$. Thus, by Lemma~{\upshape\ref{le2.8}}, we have that
\[
\begin{aligned}
    \left \| \sum\limits_{j}\mu_{j}I^{loc}_{\alpha}(b_{j}) \right \|_{L^{q}(\omega^{q})}
    &\leq\left \| \sum\limits_{j}\mu_{j}\vert I^{loc}_{\alpha}(b_{j})\vert\chi_{10P_{j}} \right \|_{L^{q}(\omega^{q})}\\
    &\leq C\left \| \sum\limits_{j}\mu_{j}\left(\frac{1}{\vert P_{j}\vert}\int_{P_{j}}\vert I^{loc}_{\alpha}(b_{j})\vert^{q_{0}}dx\right)^{\frac{1}{q_{0}}}\chi_{10P_{j}} \right \|_{L^{q}(\omega^{q})}\\
    &\leq C\left \| \sum\limits_{j}\mu_{j}\vert P_{j}\vert^{-\frac{1}{q_{0}}}\left(\int_{P_{j}}\vert b_{j}\vert^{p_{0}}dx\right)^{\frac{1}{p_{0}}}\chi_{10P_{j}} \right \|_{L^{q}(\omega^{q})}\\
    &\leq C\left \| \sum\limits_{j}\mu_{j}\frac{l(P_{j})^{\alpha_{0}}\chi_{10P_{j}}}{\omega(P_{j})^{1/p}} \right \|_{L^{q}(\omega^{q})}\\
    &\leq C\left \| \sum\limits_{j}\mu_{j}\frac{l(P_{j})^{\alpha}\chi_{10P_{j}}}{\omega(P_{j})^{1/p}} \right \|_{L^{q}(\omega^{q})}\\
    &\leq C\left \| \sum\limits_{j}\frac{\mu_{j}\chi_{P_{j}}}{\omega(P_{j})^{1/p}} \right \|_{L^{p}(\omega^{p})}\\
\end{aligned}
\]
where the third inequality follows from the boundedness of $I^{loc}_{\alpha}$ on classical Lebesgue spaces (\cite[Lemma 8.9]{2011Weighted}) and the last inequality follows from the Lemma~{\upshape\ref{le2.9}}, Remark~{\upshape\ref{re2.2}} and the fact that $\omega^{p}\in A_{\infty}$.\par
Therefore, we have proved the first part of the theorem. Now we consider the boundedness of $I^{loc}_{\alpha}$ from $h^{p}(\omega^{p})$ to $h^{q}(\omega^{q})$. To end this, we need to prove that 
\begin{equation}\label{eq3}
\left \| \sum\limits_{j}\lambda_{j}M_{\Phi}(I^{loc}_{\alpha}(a_{j})) \right \|_{L^{q}(\omega^{q})}\leq C\left \| \sum\limits_{j}\frac{\lambda_{j}\chi_{Q_{j}}}{\omega(Q_{j})^{1/p}} \right \|_{L^{p}(\omega^{p})}
\end{equation}
and
\begin{equation}\label{eq4}
\left \| \sum\limits_{j}\mu_{j}M_{\Phi}(I^{loc}_{\alpha}(b_{j})) \right \|_{L^{q}(\omega^{q})}\leq C\left \| \sum\limits_{j}\frac{\mu_{j}\chi_{P_{j}}}{\omega(P_{j})^{1/p}} \right \|_{L^{p}(\omega^{p})},
\end{equation}
where each $a_{j}$ is $\omega$\text{-}$(p,t,N)$\text{-}atom, each $b_{j}$ is $\omega$\text{-}$(p,t,N)$\text{-}block and the exact values of $t$ and $N$ will be chosen below.\par
First we prove {\upshape\ref{eq4}}. 
For $x\in\mathbb R^{n}$,
\[
\begin{aligned}
&\left\vert\sum\limits_{j}\mu_{j}M_{\Phi}(I^{loc}_{\alpha}(b_{j}))(x)\right\vert\\
&\leq\sum\limits_{j}\mu_{j}\vert M_{\Phi}(I^{loc}_{\alpha}(b_{j}))(x)\vert\chi_{P_{j}^{*}}(x)+\sum\limits_{j}\mu_{j}\vert M_{\Phi}(I^{loc}_{\alpha}(b_{j}))(x)\vert\chi_{(P_{j}^{*})^{c}}(x)\\
&=:\uppercase\expandafter{\romannumeral1}+\uppercase\expandafter{\romannumeral2}.
\end{aligned}
\]\par
To estimate $\uppercase\expandafter{\romannumeral1}$, arguing as before we may assume that $q_{0}>\max\{(\frac{r}{r-1})q,1\}$. Let $0<\alpha_{0}<\alpha<n$ satisfying $\frac{n}{\alpha_{0}}>q_{\omega^{p}}$. Define $p_{0}>q_{\omega^{p}}$ by $\frac{1}{p_{0}}-\frac{1}{q_{0}}=\frac{\alpha_{0}}{n}$. We choose $t=p_{0}$. Then by Lemma~{\upshape\ref{le2.8}}, and since $M_{\Phi}$ is bounded on $L^{q_{0}}$,
\[
\begin{aligned}
    \left \| \uppercase\expandafter{\romannumeral1} \right \|_{L^{q}(\omega^{q})}
    &\leq C\left \| \sum\limits_{j}\lambda_{j}\left(\frac{1}{P_{j}}\int_{P_{j}}\left(M_{\Phi}I^{loc}_{\alpha}(b_{j})\right)^{q_{0}}dx\right)^{\frac{1}{q_{0}}}\chi_{P_{j}^{*}} \right \|_{L^{q}(\omega^{q})}\\
    &\leq C\left \| \sum\limits_{j}\lambda_{j}\vert P_{j}\vert^{-\frac{1}{q_{0}}}\left(\int_{P_{j}}\vert I^{loc}_{\alpha}(b_{j})\vert^{q_{0}}dx\right)^{\frac{1}{q_{0}}}\chi_{P_{j}^{*}} \right \|_{L^{q}(\omega^{q})}\\
    &\leq C\left \| \sum\limits_{j}\lambda_{j}\vert P_{j}\vert^{-\frac{1}{q_{0}}}\left(\int_{P_{j}}\vert b_{j}\vert^{p_{0}}dx\right)^{\frac{1}{p_{0}}}\chi_{P_{j}^{*}} \right \|_{L^{q}(\omega^{q})}\\
    &\leq C\left \| \sum\limits_{j}\lambda_{j}\frac{\vert P_{j}\vert^{\frac{\alpha}{n}}}{\omega(P_{j})^{\frac{1}{p}}}\chi_{P_{j}^{*}} \right \|_{L^{q}(\omega^{q})}\\
    &\leq C\left \| \sum\limits_{j}\frac{\lambda_{j}\chi_{P_{j}^{*}}}{\omega(P_{j})^{\frac{1}{p}}} \right \|_{L^{p}(\omega^{p})}\\
    &\leq C\left \| \sum\limits_{j}\frac{\lambda_{j}\chi_{P_{j}}}{\omega(P_{j})^{\frac{1}{p}}} \right \|_{L^{p}(\omega^{p})}\\
\end{aligned}
\]
where the third inequality follows from the boundedness of $I^{loc}_{\alpha}$ on classical Lebesgue spaces (\cite[Lemma 8.9]{2011Weighted}), the penultimate inequality follows from Lemma\ {\upshape\ref{le2.9}} and the last inequality follows from Remark~{\upshape\ref{re2.2}}.\par
To estimate $\uppercase\expandafter{\romannumeral2}$, we choose $N$ so that 
\[
\left(\frac{n-\alpha+N+1}{n}\right)q>q_{\omega^{q}}.
\]
Let $\tau=\frac{n+N+1}{n}$. Then, since $\frac{1}{\tau p}-\frac{1}{\tau q}=\frac{\alpha}{\tau n}$, we have that
\[
1+\frac{\tau q}{(\tau p)^{'}}=\tau p\left(1-\frac{\alpha}{\tau n}\right)=\left(\frac{n-\alpha+N+1}{n}\right)q.
\]
Let $\upsilon=\omega^{\frac{1}{\tau}}$. Then we have that $\upsilon^{\tau q}=\omega^{q}\in A_{1+\frac{\tau q}{(\tau p)^{'}}}$. Equivalently, we have that $\upsilon\in A_{\tau p,\tau q}$. Therefore, by Lemma~{\upshape\ref{le2.10}} and Lemma~{\upshape\ref{le2.11}} applied to the fractional maximal operator $M_{\alpha_{\tau}}$,
\[
\begin{aligned}
    \left \| \uppercase\expandafter{\romannumeral2} \right \|_{L^{q}(\omega^{q})}
    &\leq C\left \| \left(\sum\limits_{j}\lambda_{j}(M_{\alpha_{\tau}}(\chi_{P_{j}}))^{\tau}\right)^{\frac{1}{\tau}} \right \|_{L^{q\tau}(\upsilon^{q\tau})}^{\tau}\\
    &\leq C\left \| \left(\sum\limits_{j}\lambda_{j}\frac{\chi_{P_{j}}}{\omega(P_{j})^{\frac{1}{p}}}\right)^{\frac{1}{\tau}} \right \|_{L^{p\tau}(\upsilon^{p\tau})}^{\tau}\\
    &=C\left \| \sum\limits_{j}\frac{\lambda_{j}\chi_{P_{j}}}{\omega(P_{j})^{\frac{1}{p}}} \right \|_{L^{p}(\omega^{p})}\\
\end{aligned}
\]\par
Then we prove {\upshape\ref{eq3}}. From the definition of $M_{\Phi}(I_{\alpha}^{loc}(a_{j}))$, we can obtain that 
\[
{\rm supp}(M_{\Phi}(I_{\alpha}^{loc}(a_{j})))\subset Q_{j}(c_{Q_{j}},l(Q_{j})+8)\subset 20Q_{j}.
\]
Applying the argument similar to the above proof, we can obtain the desired results.

\end{proof}


\section*{Acknowledgments}This project is supported by the National Natural Science Foundation of China (Grant No. 11901309), Natural Science Foundation of Jiangsu Province of China (Grant No. BK20180734) and Natural Science Foundation of Nanjing University of Posts and Telecommunications (Grant No. NY222168).


\bibliographystyle{amsplain}

\begin{thebibliography}{99}

\bibitem{andersen1981weighted}
K.~Andersen and R.~John, \emph{{Weighted inequalities for vector-valued maximal
  functions and singular integrals}}, Studia Mathematica \textbf{1} (1981),
  no.~69, 19--31.

  
 
\bibitem{qui1981weighted}
H.~Q. Bui, \emph{{Weighted Hardy spaces}}, Mathematische Nachrichten
  \textbf{103} (1981), no.~1, 45--62.

\bibitem{coifman1974real}
R.~Coifman, \emph{{A real variable characterization of $ {H}^{p} $}}, Studia
  Mathematica \textbf{51} (1974), no.~3, 269--274.

\bibitem{weights2011}
D.~Cruz-Uribe, J.~M. Martell, and C.~P\'erez, \emph{{Weights, extrapolation and
  the theory of Rubio de Francia}}, vol. 215 of Operator Theory: Advances and
  Applications, Birkh\"auser Basel, 2011.

\bibitem{2020A}
D.~Cruz-Uribe, K.~Moen, and H.~V. Nguyen, \emph{{A new approach to norm
  inequalities on weighted and variable Hardy spaces}}, Annales- Academiae
  Scientiarum Fennicae Mathematica \textbf{45} (2020), no.~1, 175--198.

\bibitem{ding2021dual}
W.~Ding, Y.~Han, and Y.~Zhu, \emph{{Boundedness of singular integral operators
  on local Hardy spaces and dual spaces}}, Potential Analysis \textbf{55}
  (2021), 419--441.

\bibitem{ding2019discrete}
W.~Ding, L.~Jiang, and Y.~Zhu, \emph{{Discrete Littlewood-Paley-Stein
  characterization and ${L}^2$ atomic decomposition of local Hardy spaces}},
  Acta Mathematica Sinica, English Series \textbf{35} (2019), no.~10,
  1681--1695.

\bibitem{ding2012boundedness}
Y.~Ding, Y.~Han, G.~Lu, and X.~Wu, \emph{{Boundedness of singular integrals on
  multiparameter weighted Hardy spaces}}, Potential Analysis \textbf{37}
  (2012), no.~1, 31--56.

\bibitem{2001fourier}
J.~Duoandikoetxea, \emph{{Fourier analysis}}, vol. 29 of Graduate Studies in
  Mathematics, Providence, RI, 2001.

\bibitem{fefferman1972}
C.~Fefferman and E.~M. Stein, \emph{{$H^{p}$ spaces of several variables}},
  Acta Math \textbf{129} (1972), 137--193.


\bibitem{1985Weighted}
J.~Garc\'ia-Cuerva and J.~L. Rubio~de Francia, \emph{{Weighted norm
  inequalities and related topics}}, vol. 116 of North-Holland Mathematics
  Studies, Amsterdam, 1985.

\bibitem{1985Fractional}
A.~E. Gatto, C.~E. Gutiérrez, and R.~L. Wheeden, \emph{{Fractional
  integrals on weighted $H^{p}$ spaces}}, Transactions of the American
  Mathematical Society \textbf{289} (1985), no.~2, 575--589.

\bibitem{goldberg1979local}
D.~Goldberg, \emph{{A local version of real Hardy spaces}}, Duke Math J (1979),
  27--42.

\bibitem{han1994calderon}
Y.~Han, \emph{Calder{\'o}n-type reproducing formula and the $ {T}_{b} $
  theorem}, Revista Matem{\'a}tica Iberoamericana \textbf{10} (1994), no.~1,
  51--91.

\bibitem{han1994littlewood}
Y.~Han and E.~T. Sawyer, \emph{{Littlewood-Paley theory on spaces of
  homogeneous type and the classical function spaces}}, vol. 530, Memoirs of the American Mathematical Society, 110 (1994), no. 530, vi+126 pp.


\bibitem{1978A}
R.~H. Latter, \emph{{A characterization of ${H}^{p}(\mathbb {R}^{n})$ in terms
  of atoms}}, Studia Mathematica \textbf{62} (1978), no.~1.

\bibitem{Li2021}
W.~Li and Y.~Zhu, \emph{{Atomic decomposition of weighted local Hardy spaces
  $h^{p}_{\omega}$}}, Mathematics in Practice and Theory \textbf{51} (2021),
  no.~8, 235--244.

\bibitem{Meyer1990}
Y.~Meyer, \emph{{Ondelettes et op\'erateurs.
  $\uppercase\expandafter{\romannumeral2}$}}, (French) Oprateurs de
  Calder\'on-Zygmund, Actualit\'es Math\'ematiques, Hermann, Paris, 1990.

\bibitem{nakai2012hardy}
E.~Nakai and Y.~Sawano, \emph{{Hardy spaces with variable exponents and
  generalized Campanato spaces}}, Journal of functional analysis \textbf{262}
  (2012), no.~9, 3665--3748.

\bibitem{peloso2008local}
M.~M. Peloso and S.~Secco, \emph{{Local Riesz transforms characterization of
  local Hardy spaces}}, Collectanea mathematica \textbf{59} (2008), 299--320.


\bibitem{rychkov2001littlewood}
V.S. Rychkov, \emph{{Littlewood--Paley theory and function spaces with
  ${A}^{loc}_{p}$ weights}}, Mathematische Nachrichten \textbf{224} (2001),
  no.~1, 145--180.

\bibitem{sawano2017hardy}
Y.~Sawano, K.-P Ho, D.~Yang, and S.~Yang, \emph{{Hardy spaces for ball
  quasi-Banach function spaces}}, Dissertationes Mathematicae \textbf{525}
  (2017), 1--102.

\bibitem{1961On}
E.~M. Stein, \emph{{On the theory of harmonic functions of several variables.
  $\uppercase\expandafter{\romannumeral2}$.}}, Acta Mathematica \textbf{106}
  (1961), 137--174.

\bibitem{1960On}
E.~M. Stein and G.~Weiss, \emph{{On the theory of harmonic functions of several
  variables. $\uppercase\expandafter{\romannumeral1}$.}}, Acta Mathematica
  \textbf{103} (1960), 25--62.

\bibitem{stromberg1989}
J.-O. Strömberg and A.~Torchinsky, \emph{{Weighted Hardy Spaces}}, vol. 1381
  of Lecture Notes in Mathematics, Springer Berlin, Heidelberg, 1989.

\bibitem{10.2307/2000412}
J.-O. Strömberg and R.~L. Wheeden, \emph{{Fractional Integrals on weighted
  $H^{p}$ and $L^{p}$ spaces}}, Transactions of the American Mathematical
  Society \textbf{287} (1985), no.~1, 293--321.

\bibitem{tan2019atomic}
J.~Tan, \emph{{Atomic decompositions of localized Hardy spaces with variable
  exponents and applications}}, The Journal of Geometric Analysis \textbf{29}
  (2019), 799--827.
  
\bibitem{tan21}
J.~Tan, \emph{Weighted Hardy and Carleson measure spaces estimates for fractional integrations}. Publicationes Mathematicae Debrecen {\bf 98} (2021), no. 3-4, 313--330. 

\bibitem{tan2022revisit}
J.~Tan, \emph{{A revisit to the atomic decomposition of weighted Hardy
  spaces}}, Acta Mathematica Hungarica {\bf 168}(2022), 490--508.

\bibitem{tan23}
J.~Tan, \emph{{Real-variable theory of local variable Hardy spaces}}, Acta
  Mathematica Sinica-English Series (2023),
  https://doi.org/10.1007/s10114-023-1524-0.

\bibitem{tang2012weighted}
L.~Tang, \emph{{Weighted local Hardy spaces and their applications}}, Illinois
  Journal of Mathematics \textbf{56} (2012), no.~2, 453--495.

\bibitem{1983Theory}
H.~Triebel, \emph{{Theory of function spaces}}, Monographs in Mathematics 78, Birkh\"auser Verlag, Basel, 1983.


\bibitem{wu2012atomic}
X.~Wu, \emph{{Atomic decomposition characterizations of weighted multiparameter
  Hardy spaces}}, Frontiers of Mathematics in China \textbf{7} (2012),
  1195--1212.

\bibitem{2011Weighted}
D.~Yang and S.~Yang, \emph{{Weighted local Orlicz-Hardy spaces with
  applications to pseudo-differential operators}}, Dissertationes Mathematicae
  \textbf{478} (2011), 1--78.

\bibitem{yang2012local}
D.~Yang and S.~Yang, \emph{{Local Hardy spaces of Musielak-Orlicz type and
  their applications}}, Science China Mathematics \textbf{55} (2012),
  1677--1720.



\end{thebibliography}

\enddocument

\end{document}